\documentclass[preprint,12pt]{elsarticle}
\usepackage{amsmath,amsfonts,amsthm,amscd,amssymb,graphicx}
\numberwithin{equation}{section}

\newtheorem{theorem}{Theorem}[section]

\newtheorem{lemma}[theorem]{Lemma}

\newtheorem{proposition}[theorem]{Proposition}
\newtheorem{prop}[theorem]{Proposition}
\newtheorem{corollary}[theorem]{Corollary}

\newtheorem{remark}[theorem]{Remark}

\newtheorem{rem}[theorem]{Remark}

\newcommand{\RR}{\mathbb{R}}
\newcommand{\A}{\mathbb{A}}

\newcommand{\B}{\mathbb{B}}

\newcommand{\R}{\mathbb{R}}
\newcommand{\cE}{\mathcal{E}}

\newcommand{\cG}{\mathcal{G}}
\newcommand{\cJ}{\mathcal{J}}
\newcommand{\cK}{\mathcal{K}}
\newcommand{\cL}{\mathcal{L}}
\newcommand{\cU}{\mathcal{U}}
\newcommand{\cO}{\mathcal{O}}
\newcommand{\cS}{\mathcal{S}}

\newcommand{\half}{\tfrac{1}{2}}
\newcommand{\iprod}[1]{\langle{#1}\rangle}
\newcommand{\dt}{{\frac{d}{dt}}}

\newcommand{\txi}{{\tilde \xi}}
\newcommand{\tx}{{\tilde x}}

\def\myskip{\bigskip\noindent}

\newcommand{\Div}{\mbox{div}}

\newcommand{\I}{\Im m}
\newcommand{\real}{\Re e}

\title{Multi-dimensional stability of planar Lax shocks in hyperbolic-elliptic coupled systems}

\author{Toan Nguyen{\footnote{Division of Applied Mathematics, Brown University, 182 George street, Providence, RI 02912, USA. Email: Toan\underline{~}Nguyen@Brown.edu}}
}

\begin{document}

\begin{abstract}
We study nonlinear time-asymptotic stability of small--amplitude planar Lax shocks in a model consisting of a system of multi--dimensional conservation laws coupled with an elliptic system. Such a model can be found in context of dynamics of a gas in presence of radiation. Our main result asserts that the standard uniform Evans stability condition implies nonlinear stability. The main analysis is based on the earlier developments by Zumbrun for multi-dimensional viscous shock waves and by Lattanzio-Mascia-Nguyen-Plaza-Zumbrun for one--dimensional radiative shock profiles.
\end{abstract}

\maketitle

\section{Introduction}\label{sec:intro} 

In the present paper, we consider the following general hyperbolic--elliptic coupled system,
\begin{equation}\label{modelsyst}
	\left\{\begin{array}{rc}
		u_t + \sum_{j=1}^df_j(u)_{x_j} +L\mbox{div}~q =0,\\
		-\nabla~\mbox{div}~q + q+\nabla g(u)=0,	\end{array}\right.
\end{equation}
consisting of a system of conservation laws coupled with or regularized by an elliptic system, with imposed initial data $u(x,0) = u_0(x)$. Here, $x\in \RR^d$, $L$ is a constant vector in $\RR^n$, the unknowns $u\in  \RR^n$ and $q\in \RR^d$, for $n\ge 1$, $d\ge 2$, the nonlinear vector-valued flux $f_j(u)\in \RR^n$, and the scalar function $g(u)\in \RR$. 

The study of \eqref{modelsyst} is motivated by a physical model or a so-called radiating gas model that describes dynamics of a gas in presence of radiation. Such a model (due to high-temperature effects) consists of the compressible Euler equations coupled with an elliptic system representing the radiative flux. See, for example, \cite{Hm,VK}, for its derivations and discussions further on physical applications. 

The system \eqref{modelsyst} in its spatially one-dimensional form has been extensively studied by many authors such as Liu, Schochet, and Tadmor \cite{ST,LiuTa1}, Kawashima and Nishibata \cite{KN1,KN2,KN3}, Serre \cite{Ser0,Ser1}, Ito \cite{Ito}, Lin, Coulombel, and Goudon \cite{LCG1,LCG2}, among others. 
In \cite{LMS1}, Lattanzio, Mascia, and Serre show the existence and regularity of (planar) shock profiles (whose precise definition will be recalled shortly below) in a general setting as in \eqref{modelsyst}, and recently in a collaboration with Lattanzio, Mascia, Plaza, and Zumbrun \cite{LMNPZ,NPZ}, we show that such radiative shocks with small amplitudes are nonlinearly asymptotically orbitally stable. Regarding asymptotic stability, all of aforementioned references deal with spatially one--dimensional perturbations. In this work, we are interested in {\em asymptotic stability of such a shock profile with respect to multi--dimensional perturbations.} Regarding asymptotic behaviors of solutions to the model system \eqref{modelsyst} in the multi-dimensional spaces, we mention recent related works by Wang and Wang \cite{WW} and by Liu and Kawashima \cite{LK}. There, however, the authors study stability of constant states (or the zero state) and the model system \eqref{modelsyst} that they consider is restricted to the case when $u$ are scalar functions. In this paper, we study stability of planar shocks and allow $u$ to be vector--valued functions.

\subsection{Shock profiles}\label{shockprofile} To state precisely the objective of our study, let us consider the one-dimensional system of conservation laws:
\begin{equation}\label{1d-cons-laws}u_t + f_1(u)_{x_1}=0,\end{equation} for vector function $u\in \RR^n$. 
We assume that the system is strictly hyperbolic, that is, the the Jacobian matrix $df_1(u)$ has $n$ distinct real eigenvalues $\lambda_j(u)$, $j=1,\cdots, n,$ with $\lambda_1(u)<\cdots <\lambda_n(u)$, for all $u$. It is easy to see that such a conservation laws \eqref{1d-cons-laws} admits weak solutions of the form $u = \bar u(x-st)$ with 
$$ \bar u (x)  = \left \{ \begin{array}{lrr} u_+ ,\qquad & x> x_0, \\
u_- ,\qquad & x< x_0, \end{array} \right. $$
for $u_\pm \in \RR^n$, $s\in \RR$, and $x_0 \in \RR$, assuming that the triple $(u_\pm,s)$ satisfies the Rankine-Hugoniot jump condition:
\begin{equation}
\label{RH}
f_1(u_+) - f_1(u_-) = s(u_+ - u_-).
\end{equation} 
Here, by translation invariant, we take $x_0 =0$. The triple $(u_\pm,s)$ is then called a {\bf hyperbolic shock} solution of the system \eqref{1d-cons-laws}. It is called a {\bf hyperbolic $p$-Lax shock} solution of \eqref{1d-cons-laws} if the triple further satisfies the classical $p$-Lax entropy conditions:
\begin{equation}\label{Laxcond}
\begin{aligned}
\lambda_p(u_+) < \; &s < \lambda_{p+1}(u_+), \\
\lambda_{p-1}(u_-) <\; &s < \lambda_p (u_-),
\end{aligned}
\end{equation}
 for some $p$ such that $1\le p\le n$.

Next, let us consider the one-dimensional hyperbolic-elliptic system, that is the system \eqref{1d-cons-laws} coupled with an elliptic equation:
\begin{equation}\label{1d-modelsyst}
\begin{aligned}
\left\{\begin{array}{lrr}
u_{t}+ f_1(u)_{x_1} + L q^1_{x_1} &= &0,\\
-q^1_{x_1x_1} + q^1 + g(u)_{x_1} &=&0,
\end{array}
\right.
\end{aligned}
\end{equation}
for vector functions $u\in \RR^n$ and scalar $q^1\in \RR$. Lattanzio, Mascia, and Serre have shown (\cite{LMS1}) that there exist traveling wave solutions of \eqref{1d-modelsyst} that associate with (or regularize) the hyperbolic $p$-Lax shock. To recall their result more precisely, let us denote $L_p(u)$ and $R_p(u)$ the eigenvectors of $df_1(u)$ associated to the eigenvalue $\lambda_p(u)$. Assume also that the $p^{th}$ characteristic field is genuinely nonlinear, that is,   
\begin{equation}\label{GNL} (\nabla \lambda_p)^\top \cdot R_p \neq 0, \end{equation}
and furthermore at the end states $u_\pm$, there holds the positive diffusion condition 
\begin{equation}\label{Diff-cond}
L_p(u_\pm) (L dg(u_\pm)) R_p(u_\pm) > 0.\end{equation}
Here, $dg(u_\pm)$ is the Jacobian row vector in $\RR^n$, consisting the partial derivatives in $u_j$ of $g(u)$. The condition \eqref{Diff-cond} indeed comes naturally from the Chapman-Enskog expansion, giving a right sign of the diffusion term; see, for example, \cite{ST} or \cite{LMS1}.

\bigskip

We recall the result in \cite{LMS1}: 
 
{\it Given a hyperbolic p-Lax shock $(u_\pm,s)$ of \eqref{1d-cons-laws} and the assumptions \eqref{GNL} and \eqref{Diff-cond}, there exists a traveling wave solution $(u,q^1)$ of \eqref{1d-modelsyst} with the same speed $s$ and with asymptotic constant states $(u_\pm,0)$: 
\begin{equation}\label{profile}
	(u,q^1)(x_1,t) = (U,Q^1)(x_1-st),\qquad
	(U,Q^1)(\pm\infty) = (u_\pm, 0),
\end{equation} 
Furthermore, when the shock has a sufficiently small amplitude: $|u_+ - u_- |\ll 1$, the traveling wave solution is unique (up to a translation shift) and regular (see Theorems 1.6 and 1.7 of \cite{LMS1} for precise and much more general statements). 
}

\bigskip

We call such a traveling wave \eqref{profile} {\bf a radiative $p$-Lax shock} profile. Let $Q = (Q^1,0)\in \RR^d$. It is clear that $(U,Q)$ is a particular solution to the multi-dimensional hyperbolic-elliptic system \eqref{modelsyst}, with $(U,Q^1)$ as in \eqref{profile}.  We then call the solution $(U,Q)$ {\bf the planar radiative $p$-Lax shock} of \eqref{modelsyst}. Without loss of generality (that is, by re-defining $f_1$ by $f_1 - su$), in what follows we assume that {\em the shock speed $s$ is zero.}

\bigskip

{\em In this paper, we study nonlinear time-asymptotic stability of such a planar radiative $p$-Lax shock $(U,Q)$ with sufficiently small amplitudes: $|u_+-u_-|\ll 1$}. We shall make several technical and structural assumptions. Our first set of assumptions, as a summary of the above assumptions, reads as follows:

\bigskip

{\bf (S1)} The system \eqref{1d-cons-laws} is strictly hyperbolic, and the triple $(u_\pm,0)$ is a hyperbolic $p$-Lax shock of \eqref{1d-cons-laws}. 

\bigskip

{\bf (S2)} The system \eqref{1d-modelsyst} satisfies the genuine nonlinearity and the positive diffusion conditions \eqref{GNL} and \eqref{Diff-cond}

\bigskip

By hyperbolicity, it is straightforward to see that as long as the shock profile $(U,Q)$ is smooth, it enjoys the exponential convergence to their end states, precisely,
\begin{equation}\label{layerdecay}
	\Big|(d/dx_1)^k(U-u_\pm,Q)\Big|\le C e^{-\eta |x_1|},
\end{equation}
as $|x_1|\to +\infty$, for some $\eta>0$, $k\ge 0$. See, for example, a simple proof in \cite{LMNPZ}, Lemma 2.1.

In addition, we remark that the condition \eqref{Laxcond} implies that $\lambda_p(U(x_1))$ must vanish at some point $x^0_1\in \RR$ along the shock profile. By translating $x_1$ to $x_1+ x^0_1$, we assume that it vanishes at $x_1=0$. We call such a point {\bf singular} simply because the associated ODE system obtained from the standard resolvent equations is singular at this point. For further discussions on this point, see the paragraph nearby equation \eqref{eq:firsto}. Throughout the paper, we assume that 

\bigskip

{\bf (S3)} $x_1=0$ is the unique singular point such that $\lambda_p(U(0)) =0$. Furthermore, at this point, we assume
\begin{equation}\label{non-deg}
\frac{d}{dx_1}\lambda_p(U(x_1))_{\vert_{x_1=0}} \not =0.
\end{equation}

\bigskip

The uniqueness assumption is purely for sake of simplicity. The case of finite numbers of singular points should follow similarly from our analysis.

\subsection{Structural assumptions} We shall make our second set of assumptions on structure of the system \eqref{modelsyst}. Let us recall that $df_j$ and $dg$ denote the Jacobians of the nonlinear flux functions $f_j$ and $g$, respectively. Let $\cU$ be some neighborhood in $\RR^n$ of the shock profile $U$, constructed in the previous subsection. Our next assumption concerns the symmetrizability of the system.

\bigskip

{\bf (A1)} There exists a symmetric, positive definite $A_0 = A_0(u)$ such that
$A_0 (u)df_j(u) $ is symmetric and $A_0(u)Ldg(u)$ is 
positive semi-definite, 
for all $u \in  \cU$.
\bigskip

One may notice that (A1) is a common assumption in the stability theory of conservation laws, which may go back to the original idea of Godunov and Friedrichs (see, e.g., \cite{Fri}). Essentially, by the standard symmetrizer $L^2$ or $H^s$ energy estimates, the assumption (A1) yields  the necessary local well-posedness, and is closely related to existence of an associated convex entropy of the hyperbolic system. 

\medskip

We next impose the well-known Kawashima and Shizuta (KS) condition, which has played a very crucial role in studies of time-asymptotic stability. The assumption reads  

\bigskip

{\bf (A2)} For each $\xi \in \RR^d\setminus \{0\}$, no eigenvector of
$\sum_j\xi_j df_j (u_\pm)$ lies in the
kernel of $|\xi|^2Ldg(u_\pm)$. 

\bigskip

Our use of the (KS) condition is to derive sufficient $H^s$, for large $s$, energy estimates, and therefore provide sufficient control of ``high-frequency'' part of the solution operator. Here and in what follows, by high- or low-frequency regions, we always mean the regions at the level of resolvent solutions that $|(\lambda,\tilde \xi)|$ is large or small, with $(\lambda,\tilde \xi)$ being the Laplace and Fourier transformed variables of time $t$ and the spatial variable $\tilde x$ transversal to $x_1$.

\subsection{Technical hypotheses at hyperbolic level}

Along with the above structural assumptions, we shall further make the following two technical hypotheses at the hyperbolic level (i.e., the level without the presence of $q$ in our model \eqref{modelsyst}). 

\bigskip

{\bf (H1)} The eigenvalues of $\sum_j \xi_jdf_j(u_\pm)$ have constant
multiplicity with respect to $\xi\in \RR^d$, $\xi\ne 0$.

\bigskip

{\bf (H2)} The set of branch points of the eigenvalues of $(df_1)^{-1}(i\tau + \sum_{j\ne 1} i\xi_jdf_j)(u_\pm)$,
$\tau \in \RR$, $\tilde \xi\in \RR^{d-1}$ is the (possibly
intersecting) union of finitely many smooth curves
$\tau=\eta_q^\pm(\tilde \xi)$, on which the branching eigenvalue has
constant multiplicity $s_q$ (by definition $\ge 2$).
\bigskip

These hypotheses are crucially used in our construction of the Green kernel and the resolvent solution in the low-frequency regimes, and borrowed directly from the previous
analyses introduced by Zumbrun \cite{Z3,Z4}. The condition (H1) is the standard non-strict hyperbolicity with constant multiplicity assumption. Whereas, (H2) concerns singularities of the branching eigenvalues. It provides certain compactness properties that allow us to later on perform matrix perturbations with acceptable errors. We refer the interested reader to \cite{Z3}, Section 4.1, for a thorough discussion on these conditions. In particular, (H2) is satisfied always in dimension $d = 2$ or for rotationally invariant systems in dimensions $d>2$. 

It is perhaps worthwhile to mention that these hypotheses might be weakened or dropped as observed in \cite{N1} for the case of hyperbolic-parabolic settings. More precisely, we were able to allow eigenvalues with variable multiplicities (for instance, in case of the compressible magnetohydrodynamics equations) and to drop or remove the technical condition (H2) in establishing the stability. However, we leave it for the future work, as our current purpose is to show that the well-developed stability theory \cite{Z3,Z4} for the hyperbolic-parabolic systems can be adapted into the current hyperbolic--elliptic settings despite the presence of singularity in the eigenvalue ODE systems, among other technicalities.

\bigskip

Finally, regarding regularity of the system, we make the following additional assumption: 

\bigskip

{\bf (H0)} $f_j,g, A_0 \in C^{s+1}$, for some $s$ large, $s\ge s(d)$ with $s(d):=[(d-1)/2]+5$.

\bigskip

The regularity is not optimal due to repeated use of Sobolev embeddings in our estimates of the solution operator, especially the energy-type estimate of the high-frequency solution operator in Section \ref{sec-S2}. One could lower the required regularity by deriving much more detailed description of the resolvent solution following Zumbrun \cite{Z4}, instead of using the energy-type estimate, in the high-frequency regime.   

\bigskip

Throughout the paper, to avoid repetition let us say Assumption (S) to mean the set of Assumptions (S1), S(2), and (S3); Assumption (A) for (A1) and (A2); and, Assumption (H) for (H0), (H1), and (H2).

\subsection{The Evans function condition}
As briefly mentioned in the Abstract of the paper, we prove a theorem asserting that {\em an Evans function condition implies nonlinear time-asymptotic stability of small radiative shock profiles,} under Assumptions (S), (A), and (H) mentioned earlier. Shortly below, we shall introduce the Evans function condition that is sufficient for the stability. To do so, let us formally write the system \eqref{modelsyst} in a nonlocal form:
\begin{equation}\label{modelu}\left\{\begin{array}{lc}u_t + \sum_{j=1}^df_j (u)_{x_j} -L\Div \,\cK\,\nabla g(u) =0,\\u_{|t=0} = u_0(x),
\end{array}\right.\end{equation} with $\cK:=(-\nabla \Div \cdot+ 1)^{-1}$. We then linearize the system around the shock profile $U$. The linearization formally reads
 \begin{equation}\label{linsys-intro}
    \begin{aligned}\begin{matrix}
    u_{t} - \cL u=0,\qquad\cL u:=-\sum_j(A_j(x_1)u)_{x_j} - \cJ u \end{matrix}
    \end{aligned}
\end{equation}
with initial data $u(0)=u_0$, and $\cJ u := -L\Div \,\cK\,\nabla (B(x_1)u)$. Here, we denote $A_j(x_1):=df_j(U(x_1))$ and $B(x_1):=dg(U(x_1))$.  
Hence, the Laplace--Fourier transform, with respect to variables $(t,\tx)$, $\tx$ the transversal variable, applied to equation \eqref{linsys-intro} gives
\begin{equation}
    \begin{aligned}
    \lambda u- \cL_\txi u &= S
    \end{aligned}
     \label{spectralsyst}
\end{equation}
where source $S$ is the initial data $u_0$. An evident necessary condition for stability is the absence of $L^2$ solutions for values of $\lambda$ in $\{ \real \lambda> 0\}$, 
for each $\txi\in \RR^{d-1}$,
noting that, when $\txi=0$, $\lambda = 0$ is the eigenvalue associated to translation invariance.

We establish a sufficient condition for stability, namely, the strong spectral stability condition, expressing in term of the Evans function. For a precise statement, let us denote $D_{\pm}(\lambda,\txi)$ (see their definition in \eqref{Evansfns-def} below) the two Evans functions associated with the linearized operator about the profile in regions $x_1\gtrless 0$, correspondingly. Let $\zeta = (\tilde \xi,\lambda)$. Introduce polar coordinates
$\zeta = \rho \hat \zeta$, with $\hat \zeta = (\hat{\tilde \xi},\hat
\lambda) $ on the sphere $S^d$, and  write $D_{\pm}(\lambda,\txi)$ as $D_{\pm}(\hat \zeta,\rho)$. Let us define $S^d_+ = S^d \cap \{\real\hat
\lambda \ge 0\}$. Our strong spectral (or uniform Evans) stability assumption then reads

\bigskip

{{\bf (D)} $D_\pm(\hat \zeta,\rho)$ vanishes to precisely
the first order at $\rho=0$ for all $\hat \zeta \in S^d_+$ and has no
other zeros in $S^d_+ \times \bar \RR_+$.}

\bigskip

The assumption is assumed as in the general framework of Zumbrun \cite{Z3,Z4}. Possibly, it can be verified for small-amplitudes shocks by the work of Freist\"uhler and Szmolyan \cite{FS}. It is also worth mentioning an interesting work of Plaza and Zumbrun \cite{PZ}, verifying the assumption in one-dimensional case. In addition, the assumption can also be efficiently numerically checkable; see, for example, numerical computations in \cite{HLyZ1} for the case of gas dynamics.

We remark that even though we only consider in this paper the strong form of the spectral stability assumption (D), in the same vein of the main analysis in \cite{Z3,Z4}, our results should hold for a weaker form (thus more precise description for stability), namely, the refined stability assumption which involves signs of the second derivatives of $D_\pm(\hat \zeta,\rho)$ in $\rho$. In addition, extensions to nonclassical shocks should also be possible. Nevertheless, we shall omit to carry out all these possible extensions and confine the presentation to the case of the classical Lax shocks under the strong spectral assumption (D).

\subsection{Main result}
We are now ready to state our main result. 
\medskip

\begin{theorem}\label{theo-nonstab}
Let $(U,Q)$ be the Lax radiative shock profile. Assume all Assumptions (S), (A), (H), and the strong spectral stability assumption (D). Then, the profile $(U,Q)$ with small amplitude is time-asymptotically nonlinearly stable in dimensions $d\ge 2$.

More precisely, let $(\tilde u,\tilde q)$ be the solution to \eqref{modelsyst} with initial data
$\tilde u_0$ such that the initial perturbation $u_0:=\tilde u_0 - U$ is sufficiently small in $L^1\cap H^s$,
for some $s\ge [(d-1)/2]+5$. Then $(\tilde u,\tilde q)(t)$ exists globally in time and satisfies
\begin{equation*}
	\begin{aligned}
		&|\tilde u(x,t) - U(x_1)|_{L^p}
			\le C(1+t)^{-\frac{d-1}2(1-1/p)+\epsilon}|u_0|_{L^1\cap H^s}\\
		&|\tilde u(x,t) - U(x_1)|_{H^s}
			\le C(1+t)^{-(d-1)/4}|u_0|_{L^1\cap H^s}
\end{aligned}
\end{equation*}
and
\begin{equation*}
	\begin{aligned}
		&|\tilde q(x,t) - Q(x_1)|_{W^{1,p}}
			\le C(1+t)^{-\frac{d-1}2(1-1/p)+\epsilon}|u_0|_{L^1\cap H^s}\\
		&|\tilde q(x,t) - Q(x_1)|_{H^{s+1}} \le C(1+t)^{-1/4}|u_0|_{L^1\cap H^s}
\end{aligned}
\end{equation*}
for all $p\ge 2$; here, $\epsilon>0$ is arbitrarily small in case of $d=2$, and $\epsilon=0$ when $d\ge3$.
\end{theorem}
\medbreak

We obtain the same rate of decay in time as in the case of hyperbolic--parabolic setting (see, e.g., \cite{Z4}). This is indeed due to the fact that in low-frequency regimes the estimates for the Green kernel for both cases, here for the radiative systems and there for the hyperbolic--parabolic systems, are essentially the same, away from the singular point occurring in the first-order ODE system for the former case. 
 
 Let us briefly mention the abstract framework to obtain the main theorem. First, we look at the perturbation equations with respect to perturbation variable $u = \tilde u-U$, namely,
\begin{equation}\label{pert-sys}u_t - \cL u = N(u,u_x)_x,\end{equation} where $\cL u= - \sum_j(A_j(x_1) u)_{x_j}- \cJ u$ as defined in \eqref{linsys-intro} and $N(u,u_x)$ is the nonlinear remainder term. Since $\cL u$ is a zero-order
perturbation of the generator $- \sum_j(A_j(x_1) u)_{x_j}$ of a hyperbolic equation, it generates a $C^0$ semigroup $e^{\cL t}$ on the usual $L^2$ space which enjoys the inverse Laplace-Fourier transform formulae
\begin{equation}\label{iLFT}
	\begin{aligned}
		e^{\cL t}f(x)&=\frac{1}{(2\pi i)^d} \int_{\gamma -i\infty}^{\gamma+i\infty}
		\int_{\RR^{d-1}}e^{\lambda t+i\tilde x\cdot\tilde\xi} (\lambda-\cL_\txi)^{-1} \hat f(x_1,\txi)d\txi d\lambda,
	\end{aligned}
\end{equation}where $\cL_\txi$ is the Fourier-transformed version of the operator $\cL$ in the transversal variable $\tilde x$.

Having the solution operator $e^{\cL t}$ expressed as in \eqref{iLFT}, we may now write the solution of \eqref{pert-sys} by using
Duhamel's principle as
\begin{equation}\label{Duhamel-intro}
	\begin{aligned}
		u(x,t)&= e^{\cL t}u_0(x) + \int_0^t e^{\cL(t-s)}
		N(u,u_x)_x(x,s)\, ds,
	\end{aligned}
\end{equation}
noting that $q$ can always be recovered from $u$ by $q(x,t) = -\cK\bigl(\nabla g(u)\bigr) (x,t).$ Hence, the nonlinear problem is reduced to study the solution operator at the linearized level, or more precisely, to study the resolvent solution of the resolvent equation $$(\lambda - \cL_\txi) u = f.$$

The procedure might be greatly complicated by the circumstance that the
resulting
$(n+2)\times(n+2)$
 first-order ODE system
\begin{equation}\label{eq:firsto}
(\Theta(x_1)W)_{x_1}=\A(x_1,\lambda,\txi)W, \qquad \Theta(x_1):= \begin{pmatrix} A_1(x_1) & 0 \\ 0  & I_2 \end{pmatrix},
\end{equation}
is  {\it singular} at the point where the determinant of 
$A^1(x_1)$
vanishes, with
$\Theta$ dropping from rank $n+2$ to $n+1$. However,
as already observed in \cite{LMNPZ,NPZ}, we find in the end as usual
that the Green kernel $\cG_{\lambda,\txi}$ can be constructed, and contribution of the terms due to the singular point turns out to be time-exponentially decaying.

\bigskip

The paper is organized as follows. In Section 2, we will study the resolvent solutions in low--frequency regions and define the two Evans functions, essential to the derivation of the pointwise Green kernel bounds which will be presented in Section 3. Once the resolvent bounds are obtained, estimates for the solution operator are straightforward, which will be sketched in Section 4. A damping nonlinear energy estimate is needed for nonlinear stability argument, and is derived in Section 5. In the final section, we recall the standard nonlinear argument where we use all previous linearized information to obtain the main theorem.

%


\section{Resolvent solutions and the two Evans functions}
In this section, we shall construct resolvent solutions and introduce the two Evans functions that are crucial to our later analysis of constructing the resolvent kernel. We consider the linearization of \eqref{modelsyst} around the shock profile $(U,Q)$
\begin{equation}\label{lin-sys}
	\begin{aligned}
		u_t + \sum_{j=1}^d(A_j(x_1)u)_{x_j} +L\mbox{div}~q &=0,\\
		-\nabla~\mbox{div}~q + q+\nabla (B(x_1)u)&=0,
	\end{aligned}
\end{equation} where $A_j(x_1) = df_j (U(x_1))$, $B(x_1) = dg(U(x_1))$, and $q = (q^1,q^2,\cdots, q^d) \in \RR^d$. Since the coefficients depend only $x_1$ (through $U(x_1)$), we can apply the Laplace-Fourier transform to the system \eqref{lin-sys} in time $t$ and transversal variables $\tilde x$. Let us ignore for a moment the contribution from the initial data. The Laplace-Fourier transformed system then reads 
\begin{equation}\label{LF-sys}
	\begin{aligned}
		(\lambda +i A^{\txi})u + (A_1 u)_{x_1} +Lq^1_{x_1} + iLq^\txi &=0,\\
		-(q^1_{x_1} + iq^\txi)_{x_1} + q^1+ (Bu)_{x_1}&=0,\\
	-i\xi_j(q^1_{x_1} + iq^\txi) + q^j+ i\xi_jBu&=0,\qquad j\not=1,
\end{aligned}
\end{equation} where for simplicity we have denoted $A^{\txi}:=\sum_{j\not=1}\xi_jA_j$ and $q^{\txi}:=\sum_{j\not=1}\xi_jq_j$. Multiplying the last equations by $i\xi_j$, $j\not =1$, and summing up the result, we obtain
$$(q^1_{x_1}+ iq^\txi)|\txi|^2 + iq^\txi - |\txi|^2 B u =0.$$
From this identity, we can solve $i q^\txi$ in term of $u$ and $q^1$ and then substitute it into the first two identities in the system \eqref{LF-sys}. We then obtain
\begin{equation}\label{LF-system}
	\begin{aligned}
		\Big(\lambda +i A^{\txi} + \frac{|\txi|^2}{1+|\txi|^2}LB\Big)u + (A_1 u)_{x_1} +\frac{1}{1+|\txi|^2}Lq^1_{x_1}&=0,\\
		-q^1_{x_1x_1}+ (1+|\txi|^2) q^1+ (Bu)_{x_1}&=0.
\end{aligned}
\end{equation}
System \eqref{LF-system} is a simplified and explicit version of our previous abstract form $\lambda u - \cL_\txi u =0$, where $\cL_\txi$ is defined as the Fourier transform of the linearized operator $\cL$. 


Now, by defining $$p^1:=Bu-q^1_{x_1},$$ we then easily derive the following first order ODE system from \eqref{LF-system}
\begin{equation}\label{ODEsys}
	\begin{aligned}
		(A_1 u)_{x_1} &= -(\lambda+i A^{\txi}+LB)u +(1+|\txi|^2)^{-1}Lp^1,\\
		q^1_{x_1} &=Bu - p^1,\\
	p^1_{x_1}&=-(1+|\txi|^2)q^1.
\end{aligned}
\end{equation} 
The key observation here is that this first-order ODE system is very similar to the system that we have studied for the one--dimensional case, considering the variable $\txi$ as a parameter.


\subsection{Stable/unstable dimensions} Next, we can diagonalize
$A_1$ with recalling that $A_1 (x_1)= df_1(U(x_1))$ has distinct and nonzero eigenvalues by hyperbolicity assumption (S1). Let us denote $a_p(x_1) = \lambda_p(U(x_1))$ with $\lambda_p(U)$ being the $p^{th}$ eigenvalue of $df_1(U)$, introduced in Section \ref{shockprofile}. By hyperbolicity, there exists a bounded diagonalization matrix $T(x_1)$ such that the matrix $A_1(x_1)$ can be diagonalized as follows: \begin{equation}\label{a-diag} \tilde A_1(x_1):=T^{-1} A_1 T(x_1) = \begin{pmatrix}
a_- (x_1)&&0\\&a_p(x_1)&\\0&&a_+(x_1)\end{pmatrix}
\end{equation}
where $a_-$ is the $(p-1)\times (p-1)$ matrix and negative definite, 
$a_+$ is the $(n-p)\times (n-p)$ matrix and positive definite, and $a_p\in \RR$, satisfying
$a_p(+\infty) <0< a_p(-\infty)$ (by the Lax entropy conditions \eqref{Laxcond}).

Defining $v:=T^{-1}u$, we thus obtain a diagonalized system from
\eqref{ODEsys}: 
\begin{equation}
\label{specsys-v}
\begin{aligned}
(\tilde A_1 v)_{x_1} &= - \Big(\lambda + i\tilde A^\txi+\tilde L\tilde B+(T^{-1})_{x_1} A_1 T\Big)v + (1+|\txi|^2)^{-1}\tilde Lp^1,\\
q^1_{x_1} &= \tilde B v - p^1,\\
p^1_{x_1} &= -(1+|\txi|^2)q^1,
\end{aligned}
\end{equation}
where $\tilde L:=T^{-1}L$, $\tilde B:=BT$, and $\tilde A^\txi = T^{-1} A^\txi T$.

{\em We shall construct the Green kernel for this diagonalized ODE system \eqref{specsys-v}.} To do so, let us write the system \eqref{specsys-v} in our usual matrix form with unknown $W := (v,q^1,p^1)^\top$
\begin{equation}\label{eqW}
	\left(\Theta(x_1) W\right)_{x_1} = \A(x_1,\lambda,\txi)W,
\end{equation} 
where we have denoted $\Theta(x_1):= \begin{pmatrix} \tilde A_1(x_1) & 0 \\ 0  & I_2 \end{pmatrix}$ and 
\begin{equation*}\begin{aligned}
	\A(x_1,\lambda,\txi):= \begin{pmatrix}
	 -(\lambda + i\tilde A^\txi+\tilde L\tilde B+(T^{-1})_{x_1} A_1 T) & 0 & (1+|\txi|^2)^{-1}L \\ B & 0 & -1 \\ 0  & -(1+|\txi|^2) & 0
	 	\end{pmatrix}.
\end{aligned}\end{equation*}
We note that since $a_p(0) =0$ (see the assumption (S3)), the matrix $A_1$, and thus $\Theta$, is degenerate at $x_1=0$. We shall see shortly below that this {\em singular} point causes the inconsistency in dimensions of unstable and stable manifolds, and thus the usual definition of the Evans function must be modified.     

Let us denote the limits of the coefficients as
\begin{equation}
\tilde A_\pm := \lim_{x_1\to\pm\infty}  \tilde A(x_1), \qquad \tilde B_\pm := \lim_{x_1\to\pm\infty} \tilde B(x_1), \qquad \tilde L_\pm := \lim_{x_1\to\pm\infty} \tilde L(x_1),
\end{equation}
and \begin{equation}\label{a-pm}
\A_\pm(\lambda,\txi) := \begin{pmatrix} -\tilde A_\pm^{-1}(\lambda +i\tilde A^\txi_\pm+
\tilde L_\pm\tilde B_\pm) & 0 & (1+|\txi|^2)^{-1}\tilde A_\pm^{-1}\tilde L_\pm \\ \tilde B_\pm & 0 & -1 \\
0 & -(1+|\txi|^2) & 0
\end{pmatrix}.
\end{equation}
Here, note that $\A_\pm(\lambda,\txi)$ is {\em not} quite the limiting matrix of $\A(x_1,\lambda,\txi)$ at infinities. Having defined these asymptotic matrices, the asymptotic system of \eqref{eqW} can be written as
\begin{equation}
\label{asympsyst}
 W' = \A_\pm(\lambda,\txi)W.
\end{equation}

We need to determine the dimensions of the stable/unstable eigenspaces. Observe that simple computations show $$\begin{aligned}\det (\mu - \A_{\pm}) =& \mu^2 \det (\mu+\tilde A_\pm^{-1}(\lambda +i\tilde A^\txi_\pm+
\tilde L_\pm\tilde B_\pm)) -(1+|\txi|^2) \det (\mu+\tilde A_\pm^{-1}(\lambda +i\tilde A^\txi_\pm)),\end{aligned}$$
where since for $\rho = |(\lambda,\txi)|\to 0$ the absolute value of $\tilde A_\pm^{-1}(\lambda +i\tilde A^\txi_\pm) = \cO(\rho)$, the above yields one strictly
positive and one strictly negative eigenvalues at each side of $x=\pm\infty$, denoting $\mu^\pm_1$ and $\mu^\pm_{n+2}$ (later on, giving
one fast-decaying and one fast-growing modes). Looking at slow eigenvalues $\mu = \cO(\rho)$, one easily obtains that the first term in the above computation of $\det (\mu - \A_{\pm})$ contributes $\cO(\rho^2)$ and thus eigenvalues $\mu$ are of the form \begin{equation}\label{e-values}\mu_j^\pm(\lambda,\txi) = \mu_{j0}^\pm(\lambda,\txi) + \cO(\rho^2),\end{equation} where $\mu_{j0}^\pm$ are eigenvalues of
$-\tilde A_\pm^{-1}(\lambda +i\tilde A^\txi_\pm)$. Now, notice that $\tilde A_\pm^{-1}(\lambda +i\tilde A^\txi_\pm)$ has no center subspace, i.e., no purely imaginary eigenvalue, for $\real \lambda>0$. Indeed, if it were one, say $i\xi_1$, then $\tilde A_\pm^{-1}(\lambda +i\tilde A^\txi_\pm)v=i\xi_1 v$, or equivalently, $\lambda v = - \sum_{j=1}^di\xi_j A^j_\pm v$, for some $v\in\RR^n$, which shows that $\lambda \in i\RR$ by hyperbolicity of the matrix $\sum_{j=1}^d\xi_j A^j_\pm$. Thus, $\tilde A_\pm^{-1}(\lambda +i\tilde A^\txi_\pm)$ has no center subspace. Consequently, the numbers of stable/unstable eigenvalues of $\tilde A_\pm^{-1}(\lambda +i\tilde A^\txi_\pm)$ persist as $|\txi|\to 0$, and remain the same as those of $\tilde A_\pm^{-1}$, and thus of $A^1_\pm$. We readily conclude that at $x=+\infty$, there are $p+1$
unstable eigenvalues (i.e., those with positive real parts) and $n-p+1$ stable eigenvalues (i.e., those with negative real parts) . The stable
$S^+(\lambda,\txi)$ and unstable $U^+(\lambda,\txi)$ manifolds, which consist of solutions
that decay or grow at $+\infty$, respectively, have dimensions
\begin{equation}
\label{dims+}
\begin{aligned}
\dim U^+(\lambda,\txi) &= p+1,\\
\dim S^+(\lambda,\txi) &= n-p+1,
\end{aligned}
\end{equation}
in $\Re \lambda > 0$. Likewise, there exist $n-p+2$ unstable
eigenvalues and $p$ stable eigenvalues so that the stable (solutions
which grow at $-\infty$) and unstable (solutions which decay at
$-\infty$) manifolds $S^-(\lambda,\txi)$ and $U^-(\lambda,\txi)$, respectively, have dimensions
\begin{equation}
\label{dims-}
\begin{aligned}
\dim U^-(\lambda,\txi) &=p,\\
\dim S^-(\lambda,\txi) &=n-p+2.
\end{aligned}
\end{equation}

\begin{remark}\textup{
Notice that, unlike customary situations in the Evans function
literature (see, e.g., \cite{AGJ,Z3,Z4}), the dimensions of
the stable (resp. unstable) manifolds $S^+$ and $S^-$ (resp. $U^+$
and $U^-$) \emph{do not agree}. }
\end{remark}

\subsection{Asymptotic behavior}
We study 
the asymptotic behavior of solutions to the
first order ODE system \eqref{eqW} away from the singularity point $x=0$. To simplify our presentation, we consider the case when $x \to +\infty$. Note that our treatment will be unchanged if there were finitely many singular points. We pay special attention to the small
frequency regime, $\rho \to 0$. By performing a column permutation of the last two columns in \eqref{a-pm}, with an error of order $\cO(\rho^2)$, and by further performing row reductions with observing that spectrums of the two matrices $\tilde A_+ ^{-1}(\lambda +i\tilde A^\txi_+ )$ are of order $\cO(\rho)$, strictly separated from $\pm\cO(1)$, we find that there exists a smooth matrix $V(\lambda,\txi)$ such that
\begin{equation} \label{eq-bl}V^{-1}\A_+ V =
\begin{pmatrix}H&0\\0&P\end{pmatrix}\end{equation} with blocks $P = \mbox{diag}\{P_+,P_-\} + \cO(\rho)$ with $\pm\real P_\pm>0$ and
$$\begin{aligned}H(\lambda,\txi) &= H_0(\lambda,\txi) +
\cO(\rho^2)\\H_0(\lambda,\txi)&:=-\tilde A_+ ^{-1}(\lambda +i\tilde A^\txi_+ ) = -T^{-1}A_+ ^{-1}(\lambda +i A^\txi_+ )T,\end{aligned}$$
for $T$ being the diagonalization matrix defined as in \eqref{a-diag}. We note that $H$ which determines all slow modes is spectrally equivalent to $$-A_+ ^{-1}(\lambda +i A^\txi_+ ) + \cO(\rho^2).$$

%
%
We then obtain the following lemma.
\begin{lemma}
\label{lem:asymmodes} For $\rho$ sufficiently small, the spectral system \eqref{asympsyst} associated to the
limiting, constant coefficients asymptotic behavior of
\eqref{ODEsys} has a basis of solutions
\[
e^{\mu^\pm_j(\lambda,\txi) x_1} V_j^\pm(\lambda,\txi), \quad x \gtrless 0, \:
j=1,...,n+2,
\]
where $\{V_j^\pm\}$, necessarily eigenvectors of $\A_\pm$, consist of $2n$
slow modes associated to slow eigenvalues (as in \eqref{e-values})
\begin{equation}\label{e-valuesapp}
\mu_j^\pm(\lambda,\txi) = \mu_{j0}^\pm(\lambda,\txi) + \cO(\rho^2)\qquad j=2,...,n+1,
\end{equation}
with $\mu_{j0}^\pm$ eigenvalues of
$-\tilde A_\pm^{-1}(\lambda +i\tilde A^\txi_\pm)$, and four fast modes,
\[
\begin{aligned}
\mu^\pm_1 (\lambda,\txi) &= \pm \theta^\pm_1 + \cO(\rho), \\
\mu^\pm_{n+2} (\lambda,\txi) &= \mp \theta^\pm_{n+2} + \cO(\rho).
\end{aligned}
\]where $\theta^\pm_1$ and $\theta^\pm_{n+2}$ are positive
constants.
\end{lemma}
\begin{proof} As discussed above, there are one eigenvalue with a strictly positive real part and one with a strictly negative real part at $x=\pm \infty$, giving four fast modes. Whereas, $2n$ slow modes are determined by the matrix $H$, which is spectrally equivalent to $$-A_\pm ^{-1}(\lambda +i A^\txi_\pm ) + \cO(\rho^2),$$ which gives the expansion \eqref{e-valuesapp}. Constructing the eigenvectors $V_j^\pm$ of $\A_\pm$ associated to these slow eigenvalues can be done similarly as in \cite{Z4}, Lemma 4.8, since the governing matrix $-A_\pm ^{-1}(\lambda +i A^\txi_\pm ) $ is precisely the same as those studied in the hyperbolic-parabolic systems. Note that these matrices purely come from the hyperbolic part of the system. 

The main idea of the construction is to use the assumption (H1) to separate the slow modes into  intermediate--slow (or so--called elliptic) modes for which $|\real \mu_j^\pm|\sim \rho$, super-slow (hyperbolic) modes for which $|\real \mu_j^\pm|\sim \rho^2$ and $\I\lambda$ is bounded away from any associated branch singularities $\eta_j(\txi)$, and super-slow (glancing) modes for which $|\real \mu_j^\pm|\sim \rho^2$ and $\I\lambda$ is within a small neighborhood of an associated branch singularity $\eta_j(\txi)$. Finally, thanks to the assumption (H2), the glancing blocks can also be diagonalized continuously in $\lambda$ and $\tilde \xi$, and thus associated eigenvectors can be constructed. We refer to \cite[Lemma 4.19]{Z3} for details.
 \end{proof}



In view of the structure of the asymptotic systems, we are able to
conclude that for each initial condition $x_0
> 0$, the solutions to \eqref{ODEsys} in $x_1 \geq x_0$ are spanned
by decaying/growing modes
\begin{equation}\label{phi+}\begin{aligned}\Phi^+:&=\{\phi^+_1,...,\phi^+_{n-p+1}\},
\\\Psi^+:&=\{\psi_{n-p+2}^+,...,\psi_{n+2}^+\},\end{aligned}\end{equation}
as $x_1 \to +\infty$, whereas for each initial condition $x_0 < 0$,
the solutions to \eqref{ODEsys} are spanned in $x _1\leq -x_0$ by
growing/decaying modes
\begin{equation}\label{phi-}\begin{aligned}\Psi^-:&=\{ \psi_1^-,...,
\psi^-_{n-p+2}\},\\
\Phi^-:&=\{\phi_{n-p+3}^-,...,\phi^-_{n+2}\},\end{aligned}\end{equation}
as $x _1\to -\infty$. {\em Later on, these modes will be extended on the whole line $x_1\in \RR$, by writing them as linear combinations of the corresponding modes that form a basis of solutions in respective regions $x\le -x_0$, $x\ge x_0$, or $|x|\le |x_0|$.}

We rely on the conjugation lemma of \cite{MeZ1} to link such modes
to those of the limiting constant coefficient system
\eqref{asympsyst}.
\begin{lemma}\label{lem-estmodes}\cite[Lemma 4.19]{Z3}
For $\rho$ sufficiently small, there exist unstable/stable (i.e., growing/decaying at $+\infty$ and decaying/growing at $-\infty$)
solutions $\psi^\pm_j(x_1,\lambda,\txi),\phi_j^\pm(x_1,\lambda,\txi)$, in $x_1
\gtrless \pm x_0$, of class $C^1$ in $x_1$ and continuous in $\lambda,\txi$,
satisfying
\begin{equation}\label{est-modes}
\begin{aligned}
\psi^\pm_j(x_1,\lambda,\txi) &= \gamma_{21,\psi^\pm_j}(\lambda,\txi)e^{\mu_j^\pm(\lambda,\txi)x_1} V_j^\pm(\lambda,\txi) (1
+ \cO(e^{-\eta|x_1|})), \\
\phi^\pm_j(x_1,\lambda,\txi) &= \gamma_{21,\phi^\pm_j}(\lambda,\txi)e^{\mu_j^\pm(\lambda,\txi)x_1} V_j^\pm(\lambda,\txi) (1
+ \cO(e^{-\eta|x_1|})),
\end{aligned}
\end{equation}
where $\eta>0$ is the decay rate of the traveling wave, and
$\mu_j^\pm$ and $V_j^\pm$ are as in Lemma \ref{lem:asymmodes} above. Here, the factors $$\gamma_{21,\psi^\pm_j},\gamma_{21,\phi^\pm_j}\sim 1$$ for fast and intermediate-slow modes, and for hyperbolic super-slow modes, and $$\begin{aligned}\gamma_{21,\psi^\pm_j}&\sim 1+[\rho^{-1}|\I \lambda - \eta_j^\pm(\txi)|+\rho]^{t_{\psi^\pm_j}}\\\gamma_{21,\phi^\pm_j} &\sim 1+[\rho^{-1}|\I \lambda - \eta_j^\pm(\txi)|+\rho]^{t_{\phi^\pm_j}}\end{aligned}$$
for glancing super--slow modes, for some ${t_{\phi^\pm_j}},{t_{\psi^\pm_j}} <1$ depending on $s_j$. Here, the symbol $\sim$ means that we can obtain upper and lower bounds independent of smallness of $\rho$. 
\end{lemma}

\begin{rem}
\textup{The factors $\gamma_{21,\psi^\pm_j}$ are viewed as diagonalization errors which were introduced by Zumbrun in his study of shock waves for hyperbolic/parabolic systems; see Lemma 4.19, \cite{Z3}, or Lemma 5.22, \cite{Z4}, for detailed descriptions, including, e.g., explicit computations for ${t_{\phi^\pm_j}},{t_{\psi^\pm_j}}$.}
\end{rem}


%

It will be convenient in constructing the Green kernel to define the adjoint
normal modes. Thus, let us denote
\begin{equation}\label{adjoint}\begin{pmatrix} \tilde \Psi^- & \tilde \Phi^-\end{pmatrix}
:= \begin{pmatrix} \Psi^- & \Phi^-\end{pmatrix}^{-1} \Theta^{-1}.
\end{equation}
We then obtain the following estimates.
\begin{lemma}\label{lem-estadj} For $|\rho|$ sufficiently small and $|x_1|$ sufficiently large,
\begin{equation}\label{est-adjmodes}
\begin{aligned}
\tilde \psi^\pm_j(x_1,\lambda,\txi) & =\gamma_{21,\tilde\psi^\pm_j}(\lambda,\txi) e^{-\mu_j^\pm(\lambda,\txi)x_1}\tilde V_j^\pm(\lambda,\txi)(1
+ \cO(e^{-\theta|x_1|})), \\
\tilde \phi^\pm_j(x_1,\lambda,\txi) &= \gamma_{21,\tilde\phi^\pm_j}(\lambda,\txi)e^{-\mu_j^\pm(\lambda,\txi)x_1}\tilde V_j^\pm(\lambda,\txi)(1 +
\cO(e^{-\theta|x_1|}))
\end{aligned}
\end{equation}where $\mu_j^\pm$ are defined as in Lemma
\ref{lem:asymmodes}, and $\tilde V_j^\pm$ are dual eigenvectors of $\A_\pm$. Here, as in Lemma \ref{lem-estmodes}, the factors $$\gamma_{21,\tilde\psi^\pm_j},\gamma_{21,\tilde\phi^\pm_j}\sim 1$$ for fast and intermediate-slow modes and for hyperbolic super-slow modes, and $$\begin{aligned}\gamma_{21,\tilde\psi^\pm_j}&\sim 1+[\rho^{-1}|\I \lambda - \eta_j^\pm(\txi)|+\rho]^{t_{\tilde\psi^\pm_j}}\\\gamma_{21,\tilde\phi^\pm_j} &\sim 1+[\rho^{-1}|\I \lambda - \eta_j^\pm(\txi)|+\rho]^{t_{\tilde\phi^\pm_j}}\end{aligned}$$
for glancing super--slow modes, for some ${t_{\tilde\phi^\pm_j}},{t_{\tilde\psi^\pm_j}} <1$ depending on $s_j$.
\end{lemma}
\begin{proof} The proof is straightforward from the estimates of
$\psi^\pm_j,\phi_j^\pm$ in \eqref{est-modes}.
\end{proof}

\subsection{Solutions near $x_1 \sim 0$}

Our goal now is to analyze system \eqref{ODEsys} close to the
singularity $x_1=0$. To fix ideas, let us again stick to the case
$x_1>0$, the case $x_1<0$ being equivalent. We introduce a ``stretched''
variable $\xi_1$ as follows:
\begin{equation*}
    \xi_1 = \int_{1}^{x_1}\frac{dz}{a_p(z)},
\end{equation*}
so that $\xi_1(1) = 0$, and $\xi_1 \to +\infty$ as $x_1 \to 0^+$ (note that thanks to \eqref{non-deg}, $a_p(z) \sim z$ when $z$ is small). Under
this change of variables we get
\begin{equation*}
    u' = \frac{du}{dx_1} = \frac{1}{a_p(x_1)}\frac{du}{d\xi_1} =
    \frac{1}{a_p(x_1)}\dot{u},
\end{equation*}
and denoting $\;\dot{ }$ $= d/d\xi_1$.
In the stretched variables, making some further changes of variables
if necessary, the system \eqref{specsys-v} becomes a
block-diagonalized system at leading order of the form
\begin{equation}
\label{eq:strechtedsyst} \dot{Z} = \begin{pmatrix} -\alpha & 0 \\ 0&
0 \end{pmatrix}Z + a_p(\xi_1) \B(\xi_1) Z,
\end{equation} where $\B(\xi_1)$ is some bounded matrix and
$\alpha$ is the $(p,p)$ entry of the matrix
$\lambda +i\tilde A^\txi+ \tilde L\tilde B+(T^{-1})_{x_1} A T + \tilde A_{x_1}$,
noting that due to the positive diffusion assumption (S2) on $LB$ and definitions of $\tilde L = T^{-1} L$ and $\tilde B = BT$, we have $$\real ~\alpha(\xi_1) \ge
\delta_0>0,$$ for some $\delta_0$ and any $\xi_1$ sufficiently large
or $x_1$ sufficiently near zero.

The blocks $-\alpha I$ and $0$ are clearly spectrally separated and
the error is of order $\cO(|a_p(\xi_1)|) \to 0$ as $\xi_1 \to +\infty$.
By the standard pointwise reduction lemma (see, for example, Proposition B.1, \cite{NPZ}), we can separate the flow into slow and
fast coordinates. Indeed, after proper transformations we separate
the flows on the reduced manifolds of form
\begin{align}
\dot{Z_1} &= - \alpha Z_1 + \cO(a_p) Z_1,\\
\dot{Z_2} &= \cO(a_p) Z_2.
\end{align}

Since $-\Re e\alpha \leq -\delta_0 < 0$ for $\rho \sim 0$ and $\xi
\geq 1/\epsilon$, with $\epsilon > 0$ sufficiently small, and since
$a_p(\xi_1) \to 0$ as $\xi_1 \to +\infty$, the $Z_1$ mode decay to zero
as $\xi_1 \to +\infty$, in view of
\[
e^{-\int_0^{\xi_1} \alpha(z) \, dz} \lesssim e^{-(\real (\lambda+i\tilde A^\txi) + \half
\delta_0) \xi_1}.
\]

These fast decaying modes correspond to fast decaying to zero
solutions when $x_1 \to 0^+$ in the original $u$-variable. The $Z_2$
modes comprise slow dynamics of the flow as $x_1 \to 0^+$. We summarize these into the following proposition. 

\begin{proposition}
\label{prop:smallep} \cite[Proposition 2.4]{NPZ} There exists $0 < \epsilon_0 \ll 1$
sufficiently small, such that, in the small frequency regime
$\lambda \sim 0$, the solutions to the spectral system
\eqref{ODEsys} in $(-\epsilon_0,0) \cup (0,\epsilon_0)$ are
spanned by fast modes
\begin{equation}
\label{wk} w_{k_p}^\pm(x_1,\lambda) = \begin{pmatrix} \tilde u_{k_p}^\pm
\\ \tilde q_{k_p}^\pm \\ \tilde p_{k_p}^\pm \end{pmatrix} \qquad \pm\epsilon_0 \gtrless x_1
\gtrless 0,
\end{equation}
decaying to zero as $x_1 \to 0^\pm$, and slowly varying modes
\begin{equation}
z_j^\pm(x_1,\lambda) = \begin{pmatrix} \tilde u_j^\pm
\\ \tilde q_j^\pm \\ \tilde p_j^\pm \end{pmatrix},  \qquad \pm\epsilon_0 \gtrless x_1\gtrless 0,\label{z13}
\end{equation}
with bounded limits as $x_1 \to 0^\pm$.

Moreover, the fast modes \eqref{wk} decay as
\begin{equation}
\label{decayu2} \tilde u_{{k_p}p}^\pm \sim |x_1|^{\alpha_0} \to 0
\end{equation} and
\begin{equation}
\label{decaypq2}
\begin{pmatrix} \tilde u_{{k_p}j}^\pm \\\tilde q_{k_p}^\pm \\ \tilde p_{k_p}^\pm \end{pmatrix} \sim
\cO(|x_1|^{\alpha_0}a_p(x_1)) \to 0, \qquad j\not=p,
\end{equation}
as $x_1 \to 0^\pm$; here, $\alpha_0$ is some positive constant, $k_p = n-p+2$, and $$u_{k_p}=(u_{{k_p}1},...,u_{{k_p}p},...,u_{{k_p}n})^\top.$$
\end{proposition}

\subsection{Two Evans functions}
Having constructed bases of the solutions in regions $x\le -x_0$, $x\ge x_0$, and $|x|\le |x_0|$, we can extend the modes $\phi_j^\pm$ in $\Phi^\pm$ to regions of negative/positive values of $y$ by expressing them as linear combinations of solution bases constructed in Lemmas \ref{lem-estmodes} and \ref{prop:smallep} in these respective regions. Thus, 
we are able to define the following two variable-dependent Evans functions
\begin{equation}
\label{rEvans} D_+(y_1,\lambda,\txi) := \det (\Phi^+\;W_{k_p}^{-} \;
\Phi^-)(y_1,\lambda,\txi), \qquad \mbox{for } y_1 > 0,
\end{equation} 
and 
\begin{equation}
\label{rEvans-} D_-(y_1,\lambda,\txi) := \det (\Phi^+\;W_{k_p}^{+} \;
\Phi^-)(y_1,\lambda,\txi), \qquad \mbox{for } y_1 < 0,
\end{equation} 
where $\Phi^\pm$ are defined as in \eqref{phi+}, \eqref{phi-}, and
$W_{k_p}^\pm = (u_{k_p}^\pm,q_{k_p}^\pm,p_{k_p}^\pm)^\top$ as in
\eqref{wk}, and ${k_p} = n-p+2$. 

\myskip
We observe the following simple properties of $D_\pm$.
\begin{lemma}\label{lem-Evansfns} For $\lambda$ sufficiently small, we have
\begin{equation}\label{Evansfns1}\begin{aligned}
D_\pm(y_1,\lambda,\txi)&=\gamma_\pm(y_1)(\det A_1)^{-1}
\Delta(\lambda,\txi) +
\cO(\rho^2),
\end{aligned}\end{equation}
where $\Delta (\lambda,\txi)$ is the Lopatinski determinant, defined as
\begin{equation}\label{Lop-det}\begin{aligned}\Delta (\lambda,\txi)&:= \det\begin{pmatrix}r_2^+&\cdots&r^+_{k_p-1}&r_{k_p+1}^-&\cdots&r_{n+1}^-& \lambda[u]+i[f^\txi(u)]
\end{pmatrix}
\\
\gamma_\pm(y_1)&:=\det\begin{pmatrix}q_1^+&q_{k_p}^\mp\\p_1^+&p_{k_p}^\mp\end{pmatrix}_{|_{\rho=0}}
\end{aligned}\end{equation}
with $[u] = u_+ - u_-$ and $r_j^\pm$ constant eigenvectors of $(A^1_\pm)^{-1}(LB)_\pm$, spanning the stable/unstable subspaces at $\pm\infty$, respectively.
\end{lemma}
\begin{proof} The computation follows straightforwardly from lines of the computations in \cite[pp. 59--61]{Z4}, and those in \cite[Lemma 2.5]{NPZ}.
\end{proof}

\begin{lemma}
\label{lemma-mD}
Defining the Evans functions \begin{equation}\label{Evansfns-def}
D_\pm(\lambda,\txi): = D_\pm(\pm 1, \lambda,\txi),\end{equation} we then have
\begin{equation}\label{E-consistent}
D_+(\lambda,\txi) = m D_-(\lambda,\txi) +\cO(\rho^2)\end{equation} for some nonzero factor $m$.
\end{lemma}
\begin{proof}
Proposition \ref{prop:smallep} gives
\begin{equation}
w_{k_p}^\pm(x_1) = \begin{pmatrix} \tilde u_{k_p}^\pm
\\ \tilde q_{k_p}^\pm \\ \tilde p_{k_p}^\pm \end{pmatrix} = \cO(|x|^{\alpha_0}),
\end{equation}as $x_1\to 0$, where $\alpha_0$ is defined as in Proposition \ref{prop:smallep}. Thus, there are positive constants $\epsilon_1,\epsilon_2$
near zero such that
\begin{equation*}
w_{k_p}^+(-\epsilon_1) =
w_{k_p}^-(+\epsilon_2).
\end{equation*}
Thus, this together with the fact that $w_{k_p}^\pm$ are solutions of the ODE \eqref{eqW} yields
\begin{equation*}w_{k_p}^+(-1) = m_{k_p}
w_{k_p}^-(+1)
\end{equation*}
for some nonzero constant $m_{k_p}$.
Putting these estimates into \eqref{Evansfns1} and using continuity
of $D_\pm$ in $(\lambda,\txi)$ near zero, we easily obtain the conclusion.
\end{proof}

\section{Resolvent kernel bounds in low--frequency regions} In this section, we shall
derive pointwise bounds on the resolvent kernel $\cG_{\lambda,\txi}(x_1,y_1)$
in low-frequency regimes, that is, $|(\lambda,\txi)| \to 0$, following closely the analysis developed for the one--dimensional stability in \cite{NPZ}. We recall that the Green kernel $\cG_{\lambda,\txi}(x_1,y_1)$ of the ODE system \eqref{eqW} solves
\begin{equation}\label{greeneqs}\partial_{x_1} (\Theta (x_1) \cG_{\lambda,\txi}(x_1,y_1)) - \A(x_1,\lambda,\txi)
\cG_{\lambda,\txi}(x_1,y_1) = \delta_{y_1}(x_1),
\end{equation}
in the distributional sense, where $\delta_{y_1}(\cdot)$ denotes the standard Dirac function with mass at $x_1=y_1$. The kernel $\cG_{\lambda,\txi}(.,y_1)$ then satisfies the jump conditions at $x_1=y_1$: \begin{equation}
\label{eq:conduy} [\cG_{\lambda,\txi}(.,y_1)] = \begin{pmatrix} A_1(y_1)^{-1}
&0&0\\0& 1&0\\0&0&1\end{pmatrix}.
\end{equation}
For
definiteness, throughout this section, we consider only the case
$y_1<0$. The case $y_1>0$ is completely analogous by symmetry.

By solving the jump conditions at $x_1=y_1$, one observes that the Green kernel $\cG_{\lambda,\txi}(x_1,y_1)$ can be expressed in  terms of decaying solutions at $\pm \infty$ as follows
\begin{equation}
\label{eq:formcolu} \cG_{\lambda,\txi}(x_1,y_1) = \begin{cases}
\Phi^+(x_1,\lambda,\txi)C^+(y_1,\lambda,\txi) + W_{k_p}^+(x_1,\lambda,\txi)C_{k_p}^+(y_1,\lambda,\txi), & x_1>y_1,\\
- \Phi^-(x_1,\lambda,\txi)C^-(y_1,\lambda,\txi), & x_1<y_1.\end{cases}
\end{equation}
where $C_j^\pm$ are row vectors.
We compute
the coefficients $C_j^\pm$ by means of the transmission conditions
\eqref{eq:conduy} at $y_1$. Therefore, solving by Cramer's rule the
system
\begin{equation}\label{eq:systC}
\begin{pmatrix}
\Phi^+ & W_{k_p}^+ & \Phi^-
\end{pmatrix}
\begin{pmatrix}C^+ \\ C_{k_p}^+ \\ C^-
\end{pmatrix}_{\displaystyle{|(y_1,\lambda,\txi)}} =
\begin{pmatrix} A_1(y_1)^{-1}
&0&0\\0& 1&0\\0&0&1\end{pmatrix},
\end{equation}
we readily obtain,
\begin{align}\label{C-form}
\begin{pmatrix}C^+ \\ C_{k_p}^+ \\ C^-
\end{pmatrix}(y_1,\lambda,\txi) &= D_-(y_1,\lambda,\txi)^{-1}\begin{pmatrix} \Phi^+
& W_{k_p}^+ & \Phi^-
\end{pmatrix}^{adj}\begin{pmatrix}A_1(y_1)^{-1}
&0&0\\0& 1&0\\0&0&1\end{pmatrix}
\end{align}
where $M^{adj}$ denotes the adjugate matrix of a matrix
$M$. Note that
\begin{align}
C_{jp}^\pm(y_1,\lambda,\txi) &= a_p(y_1)^{-1}D_-(y_1,\lambda,\txi)^{-1}\begin{pmatrix} \Phi^+
& W_{k_p}^+ & \Phi^-
\end{pmatrix}^{pj}(y_1,\lambda,\txi), \label{Cjp}
\\
C_{jl}^\pm(y_1,\lambda,\txi) &= \sum_k D_-(y_1,\lambda,\txi)^{-1}\begin{pmatrix} \Phi^+
& W_{k_p}^+ & \Phi^-
\end{pmatrix}^{kj}(y_1,\lambda,\txi)(A_1(y_1)^{-1})_{kl},\qquad l\not = p,\label{Cjl}
\end{align}
where $()^{ij}$ is the determinant of the $(i,j)$ minor, and
$(A_1(y_1)^{-1})_{kl}$,
$l\not=p$, are bounded in $y_1$. Thus, $C_{jp}^\pm$ are only coefficients that are
possibly singular as $y_1$ near zero because of singularity in the
$p^{th}$ column of the jump-condition matrix \eqref{eq:conduy}.

\begin{lemma}\label{lem-estCnear0} Define $\rho:=|(\lambda,\txi)|$. For $\rho$ sufficiently small and for $y_1$ near zero, we have
\begin{equation}\label{est-C13}\begin{aligned}
C_j^\pm(y_1,\lambda,\txi) &=\left\{\begin{array}{llc}\cO(\rho^{-1}),& j = 1,n+2,\\
\cO(a_p(y_1)^{-1}|y_1|^{-\alpha_0}), & j = k_p,\\
\cO(1), & \mbox{otherwise,}\end{array}\right.
\end{aligned}\end{equation}
where $k_p=n-p+2$, $\alpha_0$ is defined as in Proposition \ref{prop:smallep} and $\cO(1)$ is a uniformly
bounded function, probably depending on $y_1,\lambda,\txi$.
\end{lemma}
\begin{proof} 
We shall first estimate $C_{n+2,p}^-(y_1,\lambda,\txi)$. Observe that
$$\begin{pmatrix} \Phi^+
& W_{k_p}^+ & \Phi^-
\end{pmatrix}^{p,n+2}(y_1,\lambda,\txi) = \begin{pmatrix} \Phi^+
& W_{k_p}^+ & \Phi^-
\end{pmatrix}^{p,n+2}(y,0) + \cO(\rho)$$
where by the same way as done in Lemma \ref{lem-Evansfns} we obtain an estimate
$$\begin{pmatrix} \Phi^+
& W_{k_p}^+ & \Phi^-
\end{pmatrix}^{p,n+2}(y_1,0) =a_p(\det A_1)^{-1}\gamma_-(y_1)
\Delta^{p,n+2}(\lambda,\txi),
$$ where $\gamma_-(y_1)$ and $\Delta(\lambda,\txi)$ are defined as in \eqref{Lop-det}, and $\Delta^{p,n+2}(\lambda,\txi)$ denotes the minor determinant.
Thus, recalling
\eqref{Evansfns1}
and \eqref{Cjp}, we can estimate
$C_{n+2,p}^-(y_1,\lambda,\txi)$ as
\begin{align*}
C_{n+2,p}^-(y_1,\lambda,\txi) &= a_p(y_1)^{-1}D_-(y_1,\lambda,\txi)^{-1}\begin{pmatrix} \Phi^+
& W_{k_p}^+ & \Phi^-
\end{pmatrix}^{p,n+2}(y_1,\lambda,\txi)\\
&= - \frac1{\Delta(\lambda,\txi)}\Delta^{p,n+2} + \cO(1),
\end{align*} where $\cO(1)$ is uniformly bounded since
normal modes $\phi^\pm_j$ are all bounded
as $y_1$ is near zero. Similar computations can be done for $C_{n+2,l}^-$. Thus,
we obtain the bound for $C_{n+2}^-$ as
claimed, by our strong Evans function assumption (D).
 The bound for $C_{1}^+$ follows similarly, noting that
$\phi_{n+2}^- \equiv \phi_1^+$ at $\rho=0$.

For the estimate on $C_{k_p}^+$, we observe that by view of the definition \eqref{Evansfns1} of $D_-(y_1,\lambda,\txi)$ and
the estimate $W_{k_p}^+(y_1) \thickapprox (0,\cdots,|y_1|^\nu,\cdots,0)^t$ in Proposition \ref{prop:smallep},
\begin{equation}\label{D-est}|D_-(y_1,\lambda,\txi)| \thickapprox |y_1|^{\alpha_0}|D_-(\lambda,\txi)|.\end{equation} This together with the fact that $\phi_{n+2}^- \equiv \phi_1^+$ at
$\rho=0$ yields the estimate for $C_{k_p}^+$ as claimed.

Finally, we derive estimates for $C_j^+$ (resp. $C_j^-$) for $1<j<k_p$ (resp. $k_p<j<n+2$). By again applying the estimate \eqref{wk} on $W_{k_p}$ and using the fact that $\phi_{n+2}^- \equiv \phi_1^+$ at
$\rho=0$, we obtain $$\begin{pmatrix} \Phi^+
& W_{k_p}^+ & \Phi^-
\end{pmatrix}^{pj} = \cO(\rho|y_1|^{\alpha_0}a_p(y_1))$$ and for $k\not=p$,
$$\begin{pmatrix} \Phi^+
& W_{k_p}^+ & \Phi^-
\end{pmatrix}^{kj} = \cO(\rho|y_1|^{\alpha_0})$$
These estimates together with \eqref{D-est} and \eqref{Cjl},\eqref{Cjp} yield estimates for $C_j^\pm$ as claimed.
\end{proof}

\begin{proposition}[Resolvent kernel bounds as $|y_1|\to 0$]\label{prop-nearzero}
Let $\bar W = (U,Q)$ be the shock profile. Then, for $\rho$ sufficiently small and for $y_1$ near zero, there hold
\begin{equation}\label{G0-est1}
	\cG_{\lambda,\txi}(x_1,y_1) = \cO(\rho^{-1})(\bar W'(x_1) + \rho \cO(e^{-\eta|x_1|})) + \cO(e^{-\rho^2|x_1|})
\end{equation}
for $y_1<0<x_1$, and
\begin{equation}\label{G0-est2}
	\cG_{\lambda,\txi}(x_1,y_1) = \cO(\rho^{-1})(\bar W'(x_1) + \rho \cO(e^{-\eta|x_1|}))+
	\cO(1) \Big(1+\frac{|x|^{\nu}}{a_p(y)|y|^\nu}\Big)
\end{equation}
for $y_1<x_1<0$, and
\begin{equation*}
	\cG_{\lambda,\txi}(x_1,y_1) =\cO(\rho^{-1})(\bar W'(x_1) + \rho \cO(e^{-\eta|x_1|}))
\end{equation*}
for $x_1<y_1<0$, for some $\eta>0$.
Similar bounds can be obtained for the case $y_1>0$.
\end{proposition}
\begin{proof}
For the case $y_1<0<x_1$, by using the facts that $W_{k_p}^+(x)
\equiv 0$ and that $\phi_1^+(x_1,\lambda,\txi)$ is the fast-decaying mode at $x=+\infty$ which we can choose $\phi_1^+(x_1,0,0) = \bar W'(x_1)$, together with the estimate \eqref{est-C13}, \eqref{eq:formcolu} then becomes
\begin{equation*}\begin{aligned}
	\cG_{\lambda,\txi}(x_1,y_1) &= \Phi^+(x_1,\lambda,\txi)C^+(y_1,\lambda,\txi)= \sum_{j=1}^{k_p-1}\phi_j^+(x_1,\lambda,\txi)C^+_j(y_1,\lambda,\txi)\\&= \cO(\rho^{-1})\Big(\bar W_x +
\cO(\rho)e^{-\theta |x|}\Big) + \cO(1)\sum_{j=2}^{k_p-1}e^{\mu_j^+ x},
\end{aligned}\end{equation*}
yielding \eqref{G0-est1}, noticing that $\real \mu_j^+ \le -\theta \rho^2$ for $j = 2,\cdots, k_p-1$.
Similarly, the last case $x_1<y_1<0$ is obtained by the estimate \eqref{est-C13} and the fact that $W_3^-(x_1,\lambda,\txi)$
 is the fast-decaying mode at $x=-\infty$, and $W_3^-(x_1,0,0) = \bar W'(x_1)$. In the second case $y_1<x_1<0$, the formula
\eqref{eq:formcolu} reads
\begin{equation*}\cG_{\lambda,\txi}(x_1,y_1) =
	\Phi^+(x_1,\lambda,\txi)C^+(y_1,\lambda,\txi) +  W_{k_p}^+(x_1,\lambda,\txi)C_{k_p}^+(y_1,\lambda,\txi)
\end{equation*}	
where the first term contributes the terms as in the
first case, and the second term is estimated by \eqref{est-C13} and \eqref{wk}.
\end{proof}

Next, we estimate the kernel $\cG_{\lambda,\txi}(x_1,y_1)$ for $y_1$ away from zero. We then obtain the following representation for $\cG_{\lambda,\txi}(x_1,y_1)$, for $y_1$ large.
\begin{prop}\label{prop-greenlow}
Under the assumptions of Theorem \ref{theo-nonstab}, for $|\rho|$ sufficiently small and $|y_1|$ sufficiently large, we have
\begin{equation} \label{greenlow}
\cG_{\lambda,\txi}(x_1,y_1)=\sum_{j,k}c_{jk}^+(\lambda,\txi)\phi_j^+(x_1,\lambda,\txi)\tilde{\psi}_k^-(y_1,\lambda,\txi)^*,
\end{equation}
for $y_1<0<x_1$, and
\begin{equation} \label{greenlow+}\cG_{\lambda,\txi}(x_1,y_1)=\sum_{j,k}d^+_{jk}(\lambda,\txi)\phi_j^-(x_1,\lambda,\txi)\tilde{\psi}_k^-(y_1,\lambda,\txi)^*-
\sum_{k}\psi_k^-(x_1,\lambda,\txi)\tilde{\psi}_k^-(y_1,\lambda,\txi)^*,
\end{equation}
for $y_1<x_1<0$,  and
\begin{equation} \label{greenlow-}\cG_{\lambda,\txi}(x_1,y_1)=\sum_{j,k}d^-_{jk}(\lambda,\txi)\phi_j^-(x_1,\lambda,\txi)\tilde{\psi}_k^-(y_1,\lambda,\txi)^*+
\sum_{k}\phi_k^-(x_1,\lambda,\txi)\tilde{\phi}_k^-(y_1,\lambda,\txi)^*,
\end{equation}
for $x_1<y_1<0$, where $c_{jk}^+(\lambda,\txi),d_{jk}^\pm(\lambda,\txi)$ are
scalar meromorphic functions satisfying
$$c^+ = \begin{pmatrix}-I_{k_p}&0\end{pmatrix}\begin{pmatrix}\Phi^+ & W_{k_p}^+ & \Phi^-\end{pmatrix}^{-1}\Psi^-$$ and $$d^\pm = \begin{pmatrix}0&-I_{n-k_p}\end{pmatrix}\begin{pmatrix}\Phi^+ & W_{k_p}^+ & \Phi^-\end{pmatrix}^{-1}\Psi^-.$$
\end{prop}

\begin{proof} By using the representation \eqref{eq:formcolu} and expressing the normal modes in terms of the solutions in basis in each region $y_1>0$ or $y_1<0$, the proof follows easily by direct computations. \end{proof}

We define \begin{equation}\label{gamma}\Gamma^{\tilde \xi}:=\{\lambda~:~\real\lambda =
-\theta_1(|\tilde \xi|^2 + |\I \lambda|^2)\},\end{equation} for
$\theta_1>0$ and $|(\tilde \xi,\lambda)|$ sufficiently small. Applying Proposition \ref{prop-greenlow} and Lemmas \ref{lem-estmodes} and \ref{lem-estadj}, we obtain the following proposition.
\begin{proposition}[Resolvent kernel bounds as $|y_1|\to \infty$]\label{prop-away} For $\lambda \in \Gamma^{\tilde \xi}$ and $\rho :=|(\tilde
\xi,\lambda)|$, $\theta_1$ sufficiently small, for $|y_1|$ large enough, there holds
\begin{equation}\label{green-bound}
	\begin{aligned}
	|\partial_{y_1}^\beta\cG_{\lambda,\txi}(x_1,y_1)| \le & C\gamma_2\rho^{\beta}\Big(\rho^{-1}e^{-\theta|x_1|}e^{-\theta \rho^2|y_1|}+e^{-\theta \rho^2|x_1-y_1|}\Big),
	\end{aligned}
\end{equation}for $\beta = 0,1$, and $\gamma_2$ defined as 
\begin{equation}\label{gamma-2}
\gamma_2(\lambda,\txi):= 1+\sum_{j,\pm}[\rho^{-1}|\I \lambda - \eta_j^\pm(\txi)|+\rho]^{1/s_j-1},
\end{equation} $\eta_j^\pm,s_j$ defined as in (H2).
\end{proposition}
\begin{proof} The estimate \eqref{green-bound} is a direct consequence of the representation of $\cG_{\lambda,\txi}(x_1,y_1)$ recalled in Proposition \ref{prop-greenlow} and the estimates on the normal modes obtained in Lemmas \ref{lem-estmodes} and \ref{lem-estadj}, recalling the uniform Evans function condition gives $|D_\pm|^{-1} = \cO(\rho^{-1})$.
\end{proof}

\begin{corollary}\label{cor-resolventbd} For $\lambda \in \Gamma^{\tilde \xi}$ and $\rho :=|(\tilde
\xi,\lambda)|$, $\theta_1$ sufficiently small, there holds
\begin{equation}\label{bounds-G}
	\begin{aligned}
	|\partial_{y_1}^\beta\cG_{\lambda,\txi}(x_1,y_1)| \le & C\gamma_2\rho^{\beta}\Big(\rho^{-1}e^{-\theta|x_1|}e^{-\theta \rho^2|y_1|}+e^{-\theta \rho^2|x_1-y_1|}\Big) \\&+ \cO(1) \chi\Big(1+\frac{|x_1|^{\nu}}{a_1(y_1)|y_1|^{\nu+\beta}}\Big) , 
	\end{aligned}
\end{equation}for $\beta = 0,1$, where $\chi=1$ for $-1<y_1<x_1<0$ or $0<x_1<y_1<1$ and $\chi=0$ otherwise, and $\gamma_2$ is defined as
in \eqref{gamma-2}.\end{corollary}

\begin{remark}\label{sing-term} \textup{The last term in \eqref{bounds-G} accounts for the singularity of the Green kernel when $y_1$ is near the singular point $y_1=0$. }
\end{remark}

\section{Solution operator estimates}\label{sec-Sop}
The solution operator $\cS(t):=e^{\cL t}$
of the linearized equations may be decomposed into low frequency and high frequency parts as $\cS(t) = \cS_1(t) + \cS_2(t)$ as
in \cite{Z3}, where
\begin{equation}\label{cS1}\cS_1(t):=\frac{1}{(2\pi i)^d}\int_{|\tilde \xi|\le
r}\oint_{\Gamma^{\tilde \xi}} e^{\lambda t + i\tilde \xi \cdot\tilde
x}(\lambda-\cL_{\txi})^{-1} d\lambda d\txi\end{equation}
and 
\begin{equation}\label{formS2}\begin{aligned}\cS_2(t)f&=\frac{1}{(2\pi i)^{d}}\int_{-\theta_1-i\infty}^{-\theta_1+i\infty}
\int_{\RR^{d-1}}\chi_{|\tilde{\xi}|^{2}+|\I \lambda|^{2}\geq\theta_1}
\\
&\qquad\times
e^{i\tilde{\xi}\cdot \tilde{x}+\lambda t}
(\lambda-\cL_{\tilde{\xi}})^{-1} \hat{f}(x_1,\tilde{\xi}) d\tilde{\xi}
d\lambda,
\end{aligned}\end{equation} where we recall that \begin{equation}\Gamma^{\tilde \xi}:=\{\lambda~:~\real\lambda =
-\theta_1(|\tilde \xi|^2 + |\I \lambda|^2)\},\end{equation} for
$\theta_1>0$ sufficiently small.

Then, we obtain the following proposition.
\begin{proposition}\label{prop-estS} The solution operator $\cS(t)=e^{\cL t}$ of the linearized equations may be decomposed into low frequency and high frequency parts as $\cS(t) = \cS_1(t) + \cS_2(t)$ satisfying
\begin{equation}\label{boundcS1}
\begin{aligned} |\cS_1(t) \partial_{x_1}^{\beta_1}\partial_{\tx}^{\tilde \beta} f|_{L^p_x} \le& C
(1+t)^{-\frac{d-1}{2}(1-1/p)- \frac{|\beta|}2}|f|_{L^1_x}\\&+C(1+t)^{-\frac{d-1}{2}(1-1/p)- \frac12- \frac{|\beta'|}2}|f|_{L^{1}(\tx;H^{1+\beta_1}(x_1))}
\end{aligned}\end{equation} for all $2\le p\le \infty$, $d\ge 2$, and $\beta =
(\beta_1,\tilde\beta)$ with $\beta_1=0,1$, where $|f|_{L^{1}(\tx;H^{1+\beta_1}(x_1))}$ denotes the standard $L^1$ space in $\tx$ and the $H^{1+\beta_1}$ Sobolev space in $x_1$, and
\begin{equation}\label{boundcS2}
\begin{aligned}|\partial_{x_1}^{\gamma_1}\partial_{\tx}^{\tilde\gamma}\cS_2(t)f|_{L^2} &\le C
e^{-\theta_1t}|f|_{H^{|\gamma_1|+|\tilde \gamma|}},\end{aligned}\end{equation} for $\gamma = (\gamma_1,\tilde\gamma)$ with $\gamma_1=0,1$.
\end{proposition}

The following subsections are devoted to the proof of this proposition.

\subsection{Low--frequency estimates}\label{sec-L1Lp} Bounds on $\cS_1$ are based on the following resolvent estimates.

\begin{proposition}[Low-frequency bounds]\label{prop-resLF} Under the hypotheses of Theorem \ref{theo-nonstab}, for $\lambda \in \Gamma^{\tilde \xi}$ (defined as in \eqref{gamma}) and $\rho :=|(\tilde
\xi,\lambda)|$, $\theta_1$ sufficiently small, there holds the
resolvent bound \begin{equation}\label{res-bound} |(\cL_{\tilde
\xi}-\lambda)^{-1}\partial_{x_1}^\beta f|_{L^p(x_1)} \le
C\rho^{-1+\beta}\gamma_2| f|_{L^1(x_1)}+ C|\partial_{x_1}^\beta f|_{L^\infty(x_1)},\end{equation} for all
$2\le p\le \infty$, $\beta =0,1$, and $\gamma_2$ defined as in \eqref{gamma-2}.
\end{proposition}
\begin{proof} From the resolvent bound \eqref{bounds-G}, we obtain
\begin{equation} \notag
\begin{aligned}
|(L_{\tilde \xi}-\lambda)^{-1}\partial_{x_1}^\beta  f|_{L^p(x_1)} &=
\Big|\int \cG_{\tilde \xi,\lambda}(x_1,y_1)
  \partial_{y_1}^\beta f(y_1)\, dy_1\Big|_{L^p(x_1)}
\\
&\le
C\gamma_2\Big|\int \rho^\beta\Big(\rho^{-1}e^{-\theta|x_1|}e^{-\theta \rho^2|y_1|}+e^{-\theta \rho^2|x_1-y_1|}\Big)
|  f(y_1)| \, dy_1\Big|_{L^p(x_1)}
\\
&\qquad+C\Big|\int_{x_1}^1\Big(1+\frac{|x_1|^{\nu}}{a_1(y_1)|y_1|^{\nu}}\Big)
|  \partial_{y_1}^\beta f(y_1)| \, dy_1\Big|_{L^p(0,1)}.
\end{aligned}
\end{equation}
The first term in the first integral is estimated as $$\gamma_2\Big|e^{-\theta|x_1|}\int \rho^{-1+\beta}e^{-\theta \rho^2|y_1|}
|  f(y_1)| \, dy_1\Big|_{L^p(x_1)} \le C\gamma_2\rho^{-1+\beta} |f|_{L^1(x_1)}$$
and, by using the convolution inequality
$|g*h|_{L^p}\le |g|_{L^p}|h|_{L^1}$, the second term is bounded by
$$ C\gamma_2\rho^\beta|e^{-\theta\rho^2|\cdot|}|_{L^p(x_1)}|f|_{L^1(x_1)}\le C\gamma_2\rho^{-2/p+\beta}|f|_{L^1(x_1)}.$$
Finally, for the last term, we use the fact that $a(y_1)\sim y_1$ as $y_1\to 0$ and $\int_{x_1}^1\Big(1+\frac{|x_1|^{\nu}}{|y_1|^{\nu+1}}\Big)dy_1<+\infty$, for $x_1\in (0,1)$. The estimate \eqref{res-bound} is thus obtained as claimed.\end{proof}

\subsection{Proof of bounds for $\cS_1(t)$.}

The proof will follow in a same way as done in \cite{Z3}. We shall give a sketch here.
Let $\hat u(x_1,\txi,\lambda)$ denote the solution of
$(L_\txi-\lambda)\hat u = \hat f$, where $\hat f(x_1,\txi)$ denotes
Fourier transform of $f$, and
$$u(x,t):=\cS_1(t)f = \frac{1}{(2\pi i)^d}\int_{|\txi|\le r}\oint _{\Gamma^\txi\cap \{|\lambda|\le r\}}
e^{\lambda t+i\txi \cdot \tx}(L_\txi - \lambda)^{-1}\hat
f(x_1,\txi)d\lambda d\txi.$$

Recalling the resolvent estimates in Proposition \ref{prop-resLF},
we have
\begin{align*}|\hat u(x_1,\txi,\lambda)|_{L^p(x_1)}&\le C
\rho^{-1} \gamma_2|\hat f|_{L^1(x_1)}+C|\hat f|_{H^1(x_1)}
\\&\le C\rho^{-1}\gamma_2 |f|_{L^1(x)}+C|f|_{L^{1}(\tx;H^1(x_1))}.
\end{align*}
Therefore, using Parseval's identity, Fubini's theorem, and the triangle
inequality, we may estimate $$\begin{aligned}
|u|_{L^2(x_1,\tx)}^2(t) &= \frac{1}{(2\pi)^{2d}}\int_{x_1}
\int_{\txi}\Big|\oint_{\Gamma^\txi\cap \{|\lambda|\le r\}} e^{\lambda t}\hat
u(x_1,\txi,\lambda)d\lambda\Big|^2 d\txi dx_1
\\&=\frac{1}{(2\pi)^{2d}} \int_{\txi}\Big|\oint_{\Gamma^\txi\cap \{|\lambda|\le r\}}
e^{\lambda t}\hat u(x_1,\txi,\lambda)d\lambda\Big|^2_{L^2(x_1)}
d\txi \\&\le
\frac{1}{(2\pi)^{2d}}\int_{\txi}\Big|\oint_{\Gamma^\txi\cap \{|\lambda|\le r\}}
e^{\real\lambda t}|\hat u(x_1,\txi,\lambda)|_{L^2(x_1)}d\lambda\Big|^2
d\txi \\&\le
C|f|_{L^1(x)}^2\int_{\txi}\Big|\oint_{\Gamma^\txi\cap
\{|\lambda|\le r\}} e^{\real\lambda t}\gamma_2\rho^{-1}d\lambda\Big|^2
d\txi \\&\qquad+C|f|_{L^{1}(\tx;H^1(x_1))}^2\int_{\txi}\Big|\oint_{\Gamma^\txi\cap
\{|\lambda|\le r\}} e^{\real\lambda t}d\lambda\Big|^2
d\txi.
\end{aligned}$$

Specifically, parametrizing $\Gamma^\txi$ by $$\lambda(\txi,k) = ik
- \theta_1(k^2 + |\txi|^2), \quad k\in \RR,$$ and observing that
 by \eqref{gamma-2},
\begin{equation}\begin{aligned}\gamma_2\rho^{-1}&
\le(|k|+|\txi|)^{-1} \Big[ 1+
\sum_{j}\Big(\frac{|k-\tau_j(\txi)|}{\rho}\Big)^{1/s_j-1}\Big]\\&\le(|k|+|\txi|)^{-1}
\Big[ 1+
\sum_{j}\Big(\frac{|k-\tau_j(\txi)|}{\rho}\Big)^{\epsilon-1}\Big],\end{aligned}\end{equation}
where $\epsilon:=\frac{1}{\max_j s_j}$ with recalling that $s_j$ are defined in (H2), we estimate

$$\begin{aligned}
\int_{\txi}\Big|\oint_{\Gamma^\txi\cap \{|\lambda|\le r\}}
e^{\real\lambda t}\gamma_2\rho^{-1}d\lambda\Big|^2 d\txi &\le
\int_{\txi}\Big|\int_\RR e^{-\theta_1(k^2+|\txi|^2)
t}\gamma_2\rho^{-1}dk\Big|^2 d\txi\\&\le
\int_{\txi}e^{-2\theta_1|\txi|^2t}|\txi|^{-2\epsilon}\Big|\int_\RR
e^{-\theta_1k^2t}|k|^{\epsilon-1}dk\Big|^2 d\txi \\&\quad+\sum_j
\int_{\txi}e^{-2\theta_1|\txi|^2t}|\txi|^{-2\epsilon}\Big|\int_\RR
e^{-\theta_1k^2t}|k-\tau_j(\txi)|^{\epsilon-1}dk\Big|^2 d\txi
\\&\le
\int_{\txi}e^{-2\theta_1|\txi|^2t}|\txi|^{-2\epsilon}\Big|\int_\RR
e^{-\theta_1k^2t}|k|^{\epsilon-1}dk\Big|^2 d\txi \\&\le
Ct^{-(d-1)/2}.
\end{aligned}$$
and $$\begin{aligned}\int_{\txi}\Big|\oint_{\Gamma^\txi\cap
\{|\lambda|\le r\}} e^{\real\lambda t}d\lambda\Big|^2
d\txi &\le
\int_{\txi}\Big|\int_\RR e^{-\theta_1(k^2+|\txi|^2)
t}dk\Big|^2 d\txi\le
Ct^{-(d+1)/2}.
\end{aligned}$$

Similar estimates can be obtained for the $L^\infty$ bounds and thus the $L^p$ bounds by the standard interpolation between $L^2$ and $L^\infty$. Also, the $x_1$-derivative bounds follow similarly by using the resolvent
bounds in Proposition \ref{prop-resLF} with $\beta_1=1$. The
$\tx$-derivative bounds are straightforward by the fact that
$\widehat{\partial_{\tx}^{\tilde \beta} f} = (i\txi)^{\tilde \beta}
\hat f$.

\subsection{Proof of bounds for $\cS_2(t)$.} \label{sec-S2}
The bounds for $\cS_2(t)$ are direct consequences of the following resolvent bounds.
\begin{proposition}[High-frequency bounds]\label{prop-resHF} For some $R,C$
sufficiently large and $\theta>0$ sufficiently small,
\begin{equation}\label{oldres-est}
|(\cL_\txi - \lambda)^{-1} \hat f|_{\hat H^1(x_1)} \le C |\hat f|_{\hat
H^1(x_1)},
\end{equation}
and
\begin{equation}\label{res-est} |(\cL_\txi -
\lambda)^{-1}\hat f|_{L^2(x_1)} \le \frac{C}{|\lambda|^{1/2}} |\hat
f|_{\hat H^1(x_1)},
\end{equation}
for all $|(\txi,\lambda)|\ge R$ and $\R\lambda \ge
-\theta$, where $\hat f$ is the Fourier transform of $f$ in variable
$\tx$ and $|\hat f|_{\hat H^1(x_1)} :=
|(1+|\partial_{x_1}|+|\txi|)\hat f|_{L^2(x_1)}$.
\end{proposition}

\begin{proof} The proof is straightforward by deriving an energy estimate as a Laplace-Fourier transformed version with respect to variables $(\lambda,\tx)$ of the nonlinear damping energy estimate, presented in the next section (see, for example, an analog proof carried out in \cite{LMNPZ}, Section 6, to treat the one-dimensional problem).
\end{proof}

We also have the following:
\begin{proposition}[Mid-frequency bounds]\label{prop-resMF} Strong spectral stability (D) yields \begin{equation}
|(\cL_\txi - \lambda)^{-1}|_{\hat H^1(x_1)} \le C , \quad \mbox{for
}R^{-1}\le |(\txi,\lambda)|\le R \mbox{ and }\R\lambda \ge
-\theta,\end{equation} for any $R$ and $C=C(R)$ sufficiently large
and $\theta = \theta(R)>0$ sufficiently small, where $|\hat f|_{\hat
H^1(x_1)}$ is defined as in Proposition \ref{prop-resHF}.
\end{proposition}
\begin{proof} This is due to compactness of the set of frequencies under
consideration together with the fact that the resolvent
$(\lambda-\cL_{\tilde{\xi}})^{-1}$ is analytic with respect to $H^{1}$
in $(\tilde{\xi}, \lambda)$.\end{proof}

\begin{proof}[Proof of bounds for $\cS_2(t)$] The proof starts with the following resolvent
identity, using analyticity on the resolvent set
$\rho(\cL_\txi)$ of the resolvent $(\lambda-\cL_\txi)^{-1}$, for all
$f\in \mathcal{D}(\cL_\txi)$,
\begin{equation}\label{res-id}
(\lambda-\cL_\txi)^{-1}f=\lambda^{-1}(\lambda-\cL_\txi)^{-1}\cL_\txi
f+\lambda^{-1}f.
\end{equation}

Using this identity and \eqref{formS2}, we estimate
\begin{equation}\label{S-est}\begin{aligned}\cS_2(t)f &=\frac{1}{(2\pi i)^{d}}
\int_{-\theta_1-i\infty}^{-\theta_1+i\infty}
\int_{\RR^{d-1}}\chi_{|\tilde{\xi}|^{2}+|\I\lambda|^{2}\geq\theta_1}\\
&\qquad\qquad\times e^{i\txi\cdot\tx +\lambda
t}\lambda^{-1}(\lambda-\cL_\txi)^{-1}\cL_\txi\hat f(x_1,\txi) d\txi
d\lambda\\&\quad+\frac{1}{(2\pi i)^{d}}
\int_{-\theta_1-i\infty}^{-\theta_1+i\infty}
\int_{\RR^{d-1}}\chi_{|\tilde{\xi}|^{2}+|\I\lambda|^{2}\geq\theta_1}\\
&\qquad\qquad\times e^{i\tilde{\xi}\cdot \tilde{x} +\lambda
t}\lambda^{-1}\hat f(x_1,\tilde{\xi}) d\tilde{\xi} d\lambda\\&=:S_1
+ S_2,
\end{aligned}\end{equation}
where, by Plancherel's identity and Propositions \ref{prop-resHF} and \ref{prop-resMF}, we have
$$\begin{aligned}
|S_1|_{L^2(\tx,x_1)}&\le C
\int_{-\theta_1-i\infty}^{-\theta_1+i\infty}
|\lambda|^{-1}|e^{\lambda
t}||(\lambda-\cL_\txi)^{-1}\cL_\txi\hat f|_{L^2(\txi,x_1)}|d\lambda|
\\&\le C e^{-\theta_1 t}
\int_{-\theta_1-i\infty}^{-\theta_1+i\infty}
|\lambda|^{-3/2}\Big|(1+|\txi|)|\cL_\txi\hat f|_{H^1(x_1)}\Big|_{L^2(\txi)}|d\lambda|
\\&\le C
e^{-\theta_1t}|f|_{H^{3}_x}
\end{aligned}$$
and\begin{equation}\begin{aligned}
|S_2|_{L^2_x}&\leq\frac{1}{(2\pi )^{d}}\Big|\text{P.V.}\int_{-\theta_1-i\infty}^{-\theta_1+i\infty}\lambda^{-1}e^{\lambda t} d\lambda \int_{\RR^{d-1}}e^{i\tilde{x}\cdot\tilde{\xi}}\hat f(x_1,\tilde{\xi}) d\tilde{\xi}\Big|_{L^2_x}\\
&\qquad+ \frac{1}{(2\pi )^{d}}\Big|\text{P.V.}\int_{-\theta_1-i r}^{-\theta_1+i r}\lambda^{-1}e^{\lambda t} d\lambda \int_{\RR^{d-1}}e^{i\tilde{x}\cdot\tilde{\xi}}\hat f(x_1,\tilde{\xi}) d\tilde{\xi}\Big|_{L^2_x}\\
&\leq Ce^{-\theta_1 t} |f|_{L^2_x},
\end{aligned}\end{equation}by direct computations, noting that the integral in $\lambda$ in the first term is identically zero.
This completes the proof of the first inequality stated in the proposition.
Derivative bounds follow similarly. \end{proof}

\section{Nonlinear damping estimate}\label{sec:damping}

In this section, we establish an auxiliary damping energy estimate.
We consider the nonlinear perturbation equations for variables $(u,q)$
\begin{equation}\label{eqpert}
	\begin{aligned}
		u_{t}+ \sum_jA_j(x)u_{x_j} +L\Div q &=-\sum_jM_j(x) U_{x_1}, \\
		-\nabla \Div q + q +\nabla (B(x)u) &= 0,
	\end{aligned}
\end{equation}
where we have denoted
\begin{equation*}
	\begin{aligned}
		&A_j(x,t) := df_j (U+u),\quad  M_j(x,t) = df_j (U+u)-df_j(U),
	\end{aligned}
\end{equation*}
and $$ \quad B(x,t) := \int_0^1 dg(U(x_1)+s u(x,t)) \; ds.$$
Here, the functions $A_j(x,t)$ and $B(x,t)$ should not be confused with $A_j(x_1)$ and $B(x_1)$ that used in the previous sections. The former notation is only used in this Section.

\begin{proposition}\label{prop-damping} Under the assumptions of Theorem \ref{theo-nonstab},
so long as $\|u\|_{W^{2,\infty}}$ remains smaller than a small constant $\zeta$ and the amplitude $|U_{x_1}|$ is sufficiently small,
there holds \begin{equation}\label{damp-est}
	|u|_{H^k}^2(t)\le e^{-\eta t}|u|_{H^{k}}^2(0)+ C\int_0^t
		e^{-\eta(t-s)}|u|_{L^2}^2(s)\,ds,
		\qquad\eta>0,
\end{equation}
for $k=1,...,s$, with $s$ large as in Theorem \ref{theo-nonstab}.
\end{proposition}
\begin{proof} We symmetrize the hyperbolic system in \eqref{eqpert} as
\begin{equation}\label{eqpert-u}
	\begin{aligned}
		A_0u_{t}+ \sum_j\tilde A_j(x)u_{x_j} +A_0L\Div q &=-\sum_j\tilde M_j(x) U_{x_1}
	\end{aligned}
\end{equation}where $A_0$ is the symmetrizer matrix and $\tilde A_j = A_0 A_j$, $\tilde M_j = A_0M_j$. We then observe that \begin{equation}\label{coeff-est}|A_{0x}|,|A_{0t}|,|\tilde A_{jx}|,|\tilde A_{jt}|,|\tilde M_{jx}|,|\tilde M_{jt}|,|B_x|,|B_t| = \cO(|U_{x_1}|+\zeta).\end{equation}
Taking the inner product of $q$ against the second equation in
\eqref{eqpert} and applying the integration by parts, we easily obtain
\begin{equation*}
	|\nabla q|_{L^2}^2+|q|_{L^2}^2
	=\iprod{Bu,\nabla q}\le \half |\nabla q|^2_{L^2} +C|u|_{L^2}^2.
\end{equation*}
Likewise, we can also get for $k\geq 1$
\begin{equation}\label{estq}
	|q|_{H^k}\le C|u|_{H^{k-1}},
\end{equation} for some universal constant $C$.

Taking the inner product of $u$ against the system \eqref{eqpert-u} and integrating by parts, we get
\begin{equation*}
	\frac 12\dt \iprod{A_0u,u}=-\half\iprod{\tilde A^j_{x_j}u,u}-\iprod{U_{x_1}\tilde M_j,u}-\iprod{A_0L\Div q,u}
\end{equation*}
which together with \eqref{estq} and the H\"older inequality gives
\begin{equation}\label{u0est}
	\dt |u|_{L^2}^2 \le C |u|_{L^2}^2.
\end{equation}
Now, to obtain the estimates \eqref{damp-est} in the case of $k=1$, we compute
\begin{equation}\label{1-est}\begin{aligned}
\frac 12 \dt\iprod{A_0u_{x_k},u_{x_k}} &=\iprod{(A_0u_t)_{x_k},
u_{x_k}} + \frac12 \iprod{A_{0t}u_{x_k},u_{x_k}} - \iprod{A_{0{x_k}}u_t,u_{x_k}}\\&=-\iprod{(A_0A_ju_{x_j} + A_0L\Div q)_{x_k},
u_{x_k}} +  \iprod{\cO(|U_{x_1}|+\zeta)u_{x_k},u_{x_k}}
\end{aligned}\end{equation}
where, noting that $A_0A_j$ is symmetric,
we have $$-\iprod{A_0A_ju_{x_jx_k},u_{x_k}} = \frac 12\iprod{(A_0A)_{x_j}u_{x_k},u_{x_k}}  =  \iprod{\cO(|U_{x_1}|+\zeta)u_{x_k},u_{x_k}},$$ and $$\begin{aligned}-\iprod{(A_0L\Div q)_{x_k},
u_{x_k}} &= -\iprod{A_0L(\Div q)_{x_k},
u_{x_k}} -\iprod{(A_0L)_{x_k} \Div q,
u_{x_k}}\\ &= -\iprod{A_0LB u_{x_k},
u_{x_k}} + \iprod{\cO(|U_{x_1}|+\zeta)u_{x_k},u_{x_k}} + \|q\|_{H^1}^2\\ &= -\iprod{(A_0LB)_\pm u_{x_k},
u_{x_k}} + \iprod{\cO(|U_{x_1}|+\zeta)u_{x_k},u_{x_k}} + \cO(1)\|u\|_{L^2}^2.
\end{aligned}$$
Thus, we obtain the following first-order ``Friedrichs-type'' estimate
\begin{equation}\label{1-est01}\begin{aligned}
\frac 12 \dt\iprod{A_0u_{x_k},u_{x_k}} &= -\iprod{(A_0LB)_\pm u_{x_k},
u_{x_k}}+ \iprod{\cO(|U_{x_1}|+\zeta)u_{x_k},u_{x_k}} + \cO(1)\|u\|_{L^2}^2.
\end{aligned}\end{equation}
We quickly observe that since $LB$ is not (strongly) positive definite, the first term on the right-hand side of \eqref{1-est01} does not provide a full control on the $H^1$ norm of $u$. We shall then need to apply a so--called Kawashima-type estimate. Let us first recall the following well-known result of Shizuta and Kawashima, asserting that hyperbolic effects can compensate for degenerate diffusion $LB$, as revealed by the existence of a
compensating matrix $K$.

\begin{lemma}[Shizuta--Kawashima; \cite{KSh}]\label{lem-K} Assuming (A1), condition (A2) is equivalent to the following:

\medskip

{\bf (K1)} There exist smooth skew-symmetric ``compensating matrices''
$K(\xi)$, homogeneous degree one in $\xi$, such that
\begin{equation}\label{K1}\real\Big(A_0LB|\xi|^2-K(\xi)\sum_j\xi_jA_j \Big)_\pm\ge
\theta|\xi|^2>0\end{equation}for all $\xi \in
\RR^d\setminus\{0\}$.
\end{lemma}

We now use this lemma to give sufficient $H^1$ (or rather, $H^k$) bounds. Let $K(\xi)$ be the
skew-symmetry from the Lemma \ref{lem-K}. We then compute
\begin{equation*}\begin{aligned}
\frac12\dt\iprod{K(\partial_{x_k})u,u}
&=\iprod{Ku_t,u} +\frac12\iprod{K_tu,u} -\frac12\iprod{K_{x_k}u,u_t}
 \\&=-\iprod{KA_ju_{x_j}+KL\Div q,u} +\iprod{\cO(|U_{x_1}|+\zeta)u_{x_k},u_{x_k}} + \cO(1)\|u\|_{L^2}^2
\\&=-\iprod{KA_ju_{x_j},u}+ \iprod{\cO(|U_{x_1}|+\zeta)u_{x_k},u_{x_k}} + \cO(1)\|u\|_{L^2}^2.
\\&=-\iprod{(KA_j)_\pm u_{x_j},u}+ \iprod{\cO(|U_{x_1}|+\zeta)u_{x_k},u_{x_k}} + \cO(1)\|u\|_{L^2}^2.
\end{aligned}\end{equation*}
Using Plancherel's identity, we then obtain
\begin{equation}\label{K-est01}\begin{aligned}
\frac12\dt\iprod{K(\partial_{x})u,u}
&=\iprod{(\sum_jK(\xi)\xi_jA_j)_\pm \hat u,\hat u}+ \iprod{\cO(|U_{x_1}|+\zeta)u_{x},u_{x}} + \cO(1)\|u\|_{L^2}^2,
\end{aligned}\end{equation} where $\hat u$ is the Fourier transform of $u$ in $x$; here, $\partial_x$ stands for $\partial_{x_k}$ for {\em some $x_k$.}

Let us now combine the above estimate with the Friedrichs-type estimate. By adding up \eqref{1-est01} and \eqref{K-est01} together, we obtain
\begin{equation*}\begin{aligned}
\frac12\dt\Big(&\iprod{K(\partial_{x})u,u}+\iprod{A_0u_{x},u_{x}}\Big) \\&= -\iprod{(A_0LB|\xi|^2-K(\xi)\sum_j\xi_jA_j)_\pm \hat u,\hat u}+ \iprod{\cO(|U_{x_1}|+\zeta)u_{x},u_{x}} + \cO(1)\|u\|_{L^2}^2,
\end{aligned}\end{equation*}
which, together with \eqref{K1} and the fact that $\cO(|U_x|+\zeta)$ is sufficiently small, yields
\begin{equation}\label{keyineq1st}\begin{aligned}
\frac12\dt\Big(\iprod{K(\partial_x)u,u}+\iprod{A_0u_x,u_x}\Big)\le  -\frac12\theta\iprod{u_x,u_x}+ \cO(1)\|u\|_{L^2}^2.
\end{aligned}\end{equation}
Very similarly, we also obtain the following estimate for higher derivatives $\partial_x^\alpha$, $|\alpha|=k\ge 1$,
\begin{equation}\label{keyineqkth}\begin{aligned}
\frac12\dt\Big(\iprod{K(\partial_x)\partial_x^{\alpha-1}u,\partial_x^{\alpha-1}u}+\iprod{A_0\partial_x^{\alpha}u,\partial_x^{\alpha}u}\Big)\le  -\frac12\theta\iprod{\partial_x^{\alpha} u,\partial_x^{\alpha}u} + \cO(1)\|u\|_{H^{k-1}}^2.
\end{aligned}\end{equation}

To conclude the desired $H^k$ estimates from the above Kawashima and Friedrichs-type estimates, we define $$\cE(t):= \sum_{k=0}^s\sum_{|\alpha|=k}\delta^k\Big(\iprod{K(\partial_x)\partial_x^{\alpha-1}u,\partial_x^{\alpha-1} u}+\iprod{A_0\partial_x^{\alpha} u,\partial_x^{\alpha} u}\Big),$$
for $\delta>0$. 
By applying the standard Cauchy's inequality on $\iprod{K(\partial_x)\partial_x^{\alpha-1}u,\partial_x^{\alpha-1}u}$ and using the positive definiteness of $A_0$, we observe that $\cE(t) \sim \|u\|_{H^k}^2$. We then use the above estimates \eqref{keyineq1st} and \eqref{keyineqkth}, and take $\delta$ sufficiently small to derive
\begin{equation}\label{key-est} \dt\cE(t) \le -\theta_3 \cE(t) + C
\|u\|_{L^2}^2(t)\end{equation}for some $\theta_3>0$, from which
\eqref{damp-est} follows by the standard Gronwall's inequality. The proof of Proposition \ref{prop-damping} is then complete. 
\end{proof}

\section{Nonlinear analysis}\label{sec:nonlinear}

Defining the perturbation variable $u:= \tilde u - U$, we obtain
the nonlinear perturbation equations
\begin{equation}\label{per-eqs} u_t - \cL u = \sum_j
N^j(u,u_x)_{x_j},\end{equation} where $N^j(u,u_x)=\cO(|u||u_x|+|u|^2)$ so long as $|u|$ remains bounded. We then apply the Duhamel formula \eqref{Duhamel-intro} to \eqref{per-eqs}, yielding
\begin{equation}\label{Duhamel}
\begin{aligned}
  u(x,t)=& \cS(t) u_0 + \int_0^t \cS(t-s)\sum_j
\partial_{x_j}N^j(u,u_x)ds
\end{aligned}
\end{equation} where $u(x,0) = u_0(x)$, recalling that $\cS(t) = e^{\cL t}$ denotes the linearized solution operator.

\begin{proof}[Proof of Theorem \ref{theo-nonstab}]
Define \begin{equation}\label{zeta} \begin{aligned}\zeta(t):=\sup_{0\le s\le t}
&\Big(|u (s)|_{L^2_x}(1+s)^{\frac{d-1}4}+|u (s)|_{L^\infty_x}(1+s)^{\frac
{d-1}2-\epsilon}
\Big)\end{aligned}
\end{equation}
where $\epsilon>0$ is arbitrary small in case of $d=2$ and $\epsilon=0$ in case of $d\ge 3$. 

We first show that $\zeta(t)$ is well-defined at least locally in time. Indeed, the symmetrizability assumption (A1) easily yields the following a priori $H^s$ ``Friedrichs-type'' estimate (see also \eqref{1-est01} for an $L^2$ version):
$$ \frac{d}{dt} \|u(t)\|_{H^s}^2 \quad \le \quad C \|u(t)\|_{H^s}^2 \Big( 1 + \|u(t)\|_{H^2}\Big), $$
 for some positive constant $C$ and $s> 1+d/2$. It is then easy to see that the standard short-time theory and local well-posedness in $H^s$ can be applied for the perturbation equations \eqref{per-eqs}, from a standard nonlinear iteration scheme and the above a priori estimate. See, for example, \cite{Z7}, Proposition 1.6, for a detailed proof of the local well-posedness for symmetrizable hyperbolic and hyperbolic-parabolic systems. Furthermore, the local-wellposedness argument also shows that the solution $u \in
H^s$ indeed exists on the open time-interval for which $|u |_{H^s}$ remains
bounded, and thus on this interval $\zeta(t)$ is well-defined and
continuous. 


We shall prove next that, for all $t\ge
0$ for which the solution exists with $\zeta(t)$ uniformly bounded by
some fixed and sufficiently small constant, there holds
\begin{equation}\label{zeta-est}
\zeta(t) \le C(|u _0|_{L^1\cap H^s}+\zeta(t)^2) .\end{equation}
This bound together with continuity of $\zeta(t)$ implies that
\begin{equation}\label{zeta-est1} \zeta(t) < 2C|u _0|_{L^1\cap H^s}\end{equation}
for $t\ge0$, provided that $|u _0|_{L^1\cap H^s}< 1/4C^2$, by the standard continuous induction argument. Indeed, assume that \eqref{zeta-est1} fails. By continuity, we can take the first $T>0$ such that $\zeta(T) = 2C|u _0|_{L^1\cap H^s}$. The estimate \eqref{zeta-est} then yields
$$\begin{aligned}
 2C|u _0|_{L^1\cap H^s} = \zeta(T) &\le C\Big(|u _0|_{L^1\cap H^s}+4C^2|u _0|_{L^1\cap H^s}^2\Big) \\
&= C|u _0|_{L^1\cap H^s} \Big(1+4C^2|u _0|_{L^1\cap H^s}\Big).\end{aligned}$$
A contradiction then occurs if the initial perturbation is small, namely $|u _0|_{L^1\cap H^s} < 1/4C^2$.  

In addition, we observe that the claim also provides sufficient bounds on $H^s$ norm of the solution. To see this, we apply the Proposition \ref{prop-damping} and the Sobolev embeding inequality $|u |_{W^{2,\infty}}\le C|u |_{H^s}$.
We then have
\begin{equation}\label{Hs}\begin{aligned}|u (t)|_{H^s}^2 &\quad\le\quad Ce^{-\theta t}|u _0|_{H^s}^2
+ C \int_0^t
e^{-\theta(t-\tau)}|u (\tau)|_{L^2}^2
d\tau\\&\quad\le\quad
C(|u _0|_{H^s}^2 +\zeta(t)^2)(1+t)^{-(d-1)/2}.
\end{aligned}\end{equation}
With such a uniform bound on $H^s$ norm, the solution can then be extended to a larger time interval. Repetition of these arguments yields the global existence of the solution, provided that the claim \eqref{zeta-est} is proved uniformly in time.   
This and the estimate \eqref{zeta-est1} would then complete the proof of the main theorem.

Thus, it remains to prove the claim \eqref{zeta-est}. First by \eqref{Duhamel}, we obtain
\begin{equation}\begin{aligned} |u (t)|_{L^2}\quad\le\quad& |\cS(t)u _0|_{L^2} +
\int_0^t|\cS_1(t-s)\partial_{x_j}N^j(s)|_{L^2}ds
+ \int_0^t
|\cS_2(t-s)\partial_{x_j}N^j(s)|_{L^2}ds 
\\\quad=\quad& I_1 + I_2+I_3
\end{aligned}\end{equation}
where by using the estimates in Proposition \ref{prop-estS} we estimate
$$\begin{aligned}I_1\quad :&=\quad|\cS(t) u _0|_{L^2}\quad\le\quad C (1+t)^{-\frac{d-1}{4}}|u _0|_{L^1\cap
H^3},\\
I_2\quad:&=\quad\int_0^t|\cS_1(t-s)\partial_{x_j}N^j(s)|_{L^2}ds
\\\quad&\le\quad C\int_0^t (1+t-s)^{-\frac{d-1}{4}-\frac12}(|N^j(s)|_{L^1}+|\partial_{x_1}N^j(s)|_{L^{1}(\tx;H^1(x_1))})ds
\\\quad&\le\quad C\int_0^t (1+t-s)^{-\frac{d-1}{4}-\frac12}|u |_{H^s}^2
ds\\\quad&\le\quad C(|u _0|_{H^s}^2+\zeta(t)^2)\int_0^t
(1+t-s)^{-\frac{d-1}{4}-\frac12}(1+s)^{-\frac{d-1}{2}}
ds\\\quad&\le\quad
C(1+t)^{-\frac{d-1}{4}}(|u _0|_{H^s}^2+\zeta(t)^2)
\end{aligned}$$
and
$$\begin{aligned} I_3\quad:&=\quad\int_0^t
|\cS_2(t-s)\partial_{x_j}N^j(s)|_{L^2}ds\quad \le\quad  \int_0^t
e^{-\theta(t-s)}|\partial_{x_j}N^j(s)|_{H^3}ds
\\\quad&\le\quad C\int_0^t
e^{-\theta(t-s)}(|u |_{L^\infty} + |u _x|_{L^\infty})|u |_{H^5}ds
\quad \le\quad C\int_0^t
e^{-\theta(t-s)}|u |_{H^s}^2ds
\\\quad&\le\quad C(|u _0|_{H^s}^2+\zeta(t)^2) \int_0^t
e^{-\theta(t-s)}(1+s)^{-\frac{d-1}{2}}ds\\\quad&\le\quad
C(1+t)^{-\frac{d-1}{2}}(|u _0|_{H^s}^2+\zeta(t)^2).
\end{aligned}$$

Combining the above estimates immediately yields \begin{equation} |u (t)|_{L^2}(1+t)^{\frac{d-1}{4}} \le
C(|u _0|_{L^1\cap H^s}+\zeta(t)^2) .\end{equation}
Similarly, we can obtain estimates for $
|u (t)|_{L^{\infty}_{x}}$, noting that a Moser-type inequality (precisely, Lemma 1.5 in \cite{Z7}) is used to give: $|N(t)|_{L^\infty}\le C|u(t)|_{H^s}^2$. This then completes the proof of the claim
\eqref{zeta-est}, and therefore the main theorem.\end{proof}

\bigskip

{\bf Acknowledgements.} {\em The author thanks Kevin Zumbrun for many useful discussions throughout this work, and to Benjamin Texier and the referee for many helpful comments. He is also greatly thankful to the Foundation Sciences Math\'ematiques de Paris for their support of this work through a 2009-2010 postdoctoral fellowship.
}

\bibliographystyle{elsarticle-num}

\end{document}